\documentclass[11pt]{article}
\usepackage{latexsym}
\usepackage{amssymb}
\usepackage{amsthm}
\usepackage{amsmath}
\usepackage{color}
\usepackage{tabularx}
\usepackage{url}
\usepackage{mathtools}
\usepackage{subfigure}
\usepackage{array}
\usepackage{epsfig}
\usepackage{fullpage}
\usepackage{hyperref}
\usepackage{bm}
\usepackage{bbm}
\usepackage{scalerel}
\usepackage{soul}
\usepackage{float}
\usepackage{cleveref}
\usepackage{graphicx}
\graphicspath{{./C:/Users/lenovo/Desktop/images/}}

\usepackage{titlesec}
\titleformat{\section}{\large\bf}{\thesection}{1em}{}
\titleformat{\subsection}{\normalsize\bf}{\thesubsection}{1em}{}

\usepackage{amsmath}

\usepackage[usenames,dvipsnames]{xcolor}
\usepackage{tikz}
\definecolor{myblue}{RGB}{80,80,160}
\definecolor{mygreen}{RGB}{80,160,80}

\usepackage{titlesec}

\titleformat*{\section}{\LARGE\bfseries}
\titleformat*{\subsection}{\Large\bfseries}
\titleformat*{\subsubsection}{\large\bfseries}
\titleformat*{\paragraph}{\bfseries}
\titleformat*{\subparagraph}{\bfseries}

\newtheorem{theorem}{Theorem}[section]

\newtheorem{proposition}[theorem]{Proposition}
\newtheorem{definition}{Definition}[section]
\newtheorem{claim}[theorem]{Claim}

\newtheorem{lemma}[theorem]{Lemma}

\newtheorem{corollary}[theorem]{Corollary}

\newtheorem*{remark}{Remark}
\newtheorem*{example}{Example}
\newtheorem{observation}[theorem]{Observation}
\newtheorem*{note}{Note}

\def\cR{{\cal R}}

\def\bbF{{\mathbb F}}

\def\bbN{{\mathbb N}}

\def\bbR{{\mathbb R}}

\def\bbZ{{\mathbb Z}}
\def\bbA{{\mathbb A}}

\newcommand*{\bigcircledast}{\mathop{\scalerel*{\circledast}{\big(}}\limits}

\newcommand{\cupdot}{\mathbin{\mathaccent\cdot\cup}}

\newcommand{\abs}[1]{\left\vert {#1} \right\vert}

\makeatletter
\newcommand{\ostar}{\mathbin{\mathpalette\make@circled p}}
\newcommand{\make@circled}[2]{%
  \ooalign{$\m@th#1\smallbigcirc{#1}$\cr\hidewidth$\m@th#1#2$\hidewidth\cr}%
}
\newcommand{\smallbigcirc}[1]{%
  \vcenter{\hbox{\scalebox{0.77778}{$\m@th#1\bigcirc$}}}%
}
\makeatother

\newcommand{\remove}[1]{}

\title{\bf  Hyper-Regular Graphs and High Dimensional Expanders}

\author{Ehud Friedgut \\
\small Faculty of Mathematics and Computer Science, Weizmann Institute of Science, Rehovot, Israel\\[-0.8ex]
\small\tt ehud.friedgut@weizmann.ac.il\\
\and
Yonatan Iluz \\
\small  Faculty of Mathematics and Computer Science, Weizmann Institute of Science, Rehovot, Israel \\[-0.8ex]
\small\tt yonatan.iluz@weizmann.ac.il 
\\
 {\em Dedicated to Nati Linial, a human high dimensional expander, on occasion of his 70th birthday}
}

\begin{document}

\maketitle

\begin{abstract}
Let $G= (V,E)$ be a finite graph. For $d_0>0$ we say that $G$ is $d_0$-regular, if every $v\in V$ has degree $d_0$. We say that $G$ is $(d_0, d_1)$-regular, for $0<d_1<d_0$, if $G$ is $d_0$ regular and for every $v\in V$, the subgraph induced on $v$'s neighbors is $d_1$-regular. Similarly, $G$ is $(d_0, d_1,\ldots, d_{n-1})$-regular for $0<d_{n-1}<\ldots<d_1<d_0$, if $G$ is $d_0$ regular and for every $v\in V$, the subgraph induced on $v$'s neighbors is $(d_1,\ldots, d_{n-1})$-regular (i.e. for every $1\leq i\leq n-1$, the joint neighborhood of every clique of size $i$ is $d_i$-regular); In that case, we say that $G$ is  an $n$-dimensional hyper-regular graph (HRG).
Here we define a new kind of graph product, through which we build examples of infinite families of $n$-dimensional HRG such that the joint neighborhood of every clique of size at most $n-1$ is connected. In particular, relying on the work of Kaufman and Oppenheim, our product yields an infinite family of $n$-dimensional HRG for arbitrarily large $n$ with good expansion properties. This answers a question of Dinur regarding the existence of such objects.

\end{abstract}

\section{Introduction}
Let $G= (V,E)$ be a finite graph. For $d_0 > 0$ we say that $G$ is $d_0$-regular, if every $v\in V$ has degree $d_0$. We say that $G$ is $(d_0, d_1)$-regular, for $0<d_1<d_0$, if $G$ is $d_0$ regular and for every $v\in V$, the subgraph induced on $v$'s neighbors is $d_1$-regular. Similarly, $G$ is $(d_0, d_1,\ldots, d_{n-1})$-regular for $0<d_{n-1}<\ldots<d_1<d_0$, if for every $1\leq i\leq n-1$, the joint neighborhood of every clique of size $i$ is $d_i$-regular.
\begin{example}
The complete graph on $n$ vertices $K_n$ is $(n-1, n-2,\ldots,1)$-regular, and by taking graphs consisting of arbitrarily many disjoint copies of $K_n$ one has an infinite family of $(n-1,n-2,\ldots,1)$-regular graphs.
\end{example}
While $K_n$ is connected and has good expansion properties for fixed $n$, if one turns it into an infinite family of graphs by taking disjoint copies of it, the connectivity (and of course the expansion) are lost. The challenge is finding infinite families of connected graphs, where the links have fixed degree and are also connected. An additional desired property beyond connectivity is good expansion.
Expansion of a graph requires that it is simultaneously sparse and highly connected.
The property of being an expander is significant in many contexts; In theoretical computer science, expanders are useful in the design and analysis of communication networks, and have surprising utilities in other computational settings such as in the theory of error correcting codes and the theory of pseudo-randomness. A very pleasing example of
such an application is Dinur's proof of the PCP Theorem \cite{RS07}. Now, recent advances in PCP theory \cite{DK17} require more specialized expander graphs, in particular, their construction uses highly regular hypergraphs, and could be simplified by using an infinite family of $(a,b,c,d)$-regular graphs which are expanders. However, until now, no such graphs with link-connectivity existed. A word regarding the importance of the graphs of higher regularity for the construction of HDX, such as those in  \cite{DK17}; The transition operators in those constructions, which move between various levels of the underlying hypergraph ("up-down'' and "down-up'' walks) necessitate weights on the links of the vertices. In the case of a hypergraph whose 1-dimensional skeleton is a hyper-regular graph this turns into a simple, non-weighted random walk, which makes the whole analysis much simpler.

When studying hyper-regular expanders, the first obstacle we encounter, relative to the $1$-dimensional case, is the difficulty to generate non-explicit (probabilistic) constructions: While for every $a > 2$ asymptotically almost every $a$-regular graph is a good expander, it is easy to verify that almost every $a$-regular graph is very far from being $(a,b)$-regular, as the neighborhood of a vertex is typically an anticlique.

During the special year (2016-2017) devoted to high dimensional expanders, at the Institute for Advanced Studies in the Hebrew University, Irit Dinur raised the question whether one can find a purely combinatorial (non-algebraic) construction of high dimensional expanders. As a first step in this direction, Chapman, Linial and Peled \cite{chapman2018expander} designed infinite families of $(a,b)$-regular graphs, and asked the question which we address in this paper: do their exist arbitrarily long sequences $(d_0,\ldots,d_r)$ and infinite families of connected (and preferably well expanding) 
$(d_0,\ldots,d_r)$-regular graphs?

\paragraph*{Previous work}
\begin{enumerate}
   
    \item Chapman-Linial-Peled: Families of $(a,b)$-regular graphs that expand both locally and globally (Polygraph constructions). \cite{chapman2018expander}
    \item Kaufman-Oppenheim: Bounded degree simplicial complexes arising from elementary matrix groups, which are high dimensional expanders obeying the local spectral expansion property. In their fundamental example the 1-dimensional skeletons of this family are $(a,b)$-regular. \cite{kaufman2017simplicial}
     \item In a public lecture of the second author, \cite{Ehud}, he observed that the skeleton of a tiling of the 4-dimensional hyperbolic space (namely, the order-5 5-cell four-dimensional hyperbolic honeycomb) is (120,12,5,2)-regular, with connected links.
      Taking finite quotients ( which can be found using the fundamental group of a manifold created by gluing faces of the fundamental domain of the tessellation), one gets an infinite family of $(120,12,5,2)$-regular graphs. Until the current work this was the only known example of a $(d_0,\ldots, d_r)$-HRG for $r>2$.
    \item In \cite{CLST22} the authors (a) Independently discover the $(120,12,5,2)$-regular family mentioned above, and calculate its expansion. (b) Manage to create an infinite families of $(d_0,\ldots,d_r)$-HRG  with good expansion for arbitrarily large $r$, where all but one level of the links are connected.

\end{enumerate}

\paragraph*{Main results} As mentioned above, up until now, the $(120,12,5,2)$-regular family was the only known family of $(a,b,c,d)$-regular graphs that are connected, and all other known families of connected HRG were $(a,b)$-regular.
In this paper we introduce two main examples (and tools for construction of arbitrarily more) of infinite families of $(d_0, d_1,\ldots, d_{n-1})$-regular connected graphs for $0<d_{n-1}<\ldots<d_1<d_0$ and for arbitrarily large $n$. Our construction is based on a multi-partite variant of a standard graph product and a symmetrization trick. The method should be applied to graphs that have certain symmetries such as those found in a subgroup geometry system. We prove that the method preserves connectivity and expansion properties, while regularizing the graph. In particular, relying on the work of Kaufman and Oppenheim on subgroup geometry systems \cite{kaufman2017simplicial}, one of the examples has good expansion properties. The other example is also a subgroup geometry system, arising from  symmetries of triangulation of Euclidean space (and finite tori). 

\begin{theorem}
For every $\lambda>0$, $n\in \bbN$ there exists an infinite family of $(d_0,d_1,\ldots,d_{n-1})$-regular graphs which are 1-dimensional skeletons of $\lambda$-high dimensional expanders, (namely their $(n-1)$-dimensional clique-complex).
\end{theorem}
(See definition \ref{hdx} for the definition of  $\lambda$- high dimensional expanders.)
\paragraph*{Organization} 
We start with defining basic terms in hypergraph theory, and hyper-regularity. In \ref{sub:partite} we define the multi-partite graph product and show that it preserves connectivity and link connectivity when applied on pure clique complexes. Then, we present in \ref{sub:sym} the symmetrization action and prove that graphs having  a certain property (type-regular) can be fully regularized using the symmetrization and the multi-partite product. We complete this section in \ref{sub:spectral} by showing that the expansion is also preserved under the product of graphs of this kind (pure, type-regular graphs). \\
In section \ref{sec:sgs} we show that hypergraphs arising from subgroup geometry systems are in particular pure, type-regular connected where all links are connected as well. I.e., any such hypergraph can be fully regularized by our method to form a strongly gallery connected (\ref{def:sgc}) HRG.\\
In section \ref{sec:method} we show a general scheme of generating an infinite family of finite HRG given an infinite subgroup geometry system, by taking appropriate quotients.\\
In section \ref{sec:constructions} we present two main constructions of HRG, both arising from subgroup geometry systems, where one of them has good expansion properties. We continue by investigating several interesting ad-hoc variants of those constructions. We end with the calculations of the degrees of the main constructions shown in that section.   
\subsection{Hypergraphs and simplicial complexes}
\begin{figure}
    \centering
    \includegraphics[scale=0.4]{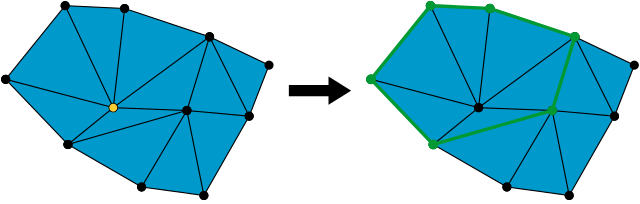}
    \caption{A vertex and its link. The black lines and the blue triangles are the $1$-dimensional and $2$-dimensional faces, respectively.}
    \label{fig:link}
\vspace{1cm}
\end{figure}
\begin{definition}[Hypergraph]
A hypergraph is a pair $(V, E)$ with $V$ a set of vertices and $E$ a set of subsets of $V$.
\end{definition}

\begin{definition}[Simplicial complex] A simplicial complex is a hypergraph that is downwards closed with respect to containment.
Namely, if $S\in E$ and $S'\subset S$, then $S'\in E$. $S\in E$ is called a face. We usually partition a simplicial complex $X$ to $X = X(0) \cupdot X(1) \cupdot \ldots \cupdot X(d)$ where $X(i)$ is the set of all faces of size $i+1$ (dimension $i$). \\
We say that a simplicial complex is $d$-dimensional if the maximal face size is $d + 1$.\\
A $d$-dimensional simplicial complex is \textbf{pure} if every face is contained in a $d$-dimensional
face. 
\end{definition}

\begin{definition}[Clique complex]
The clique complex $X(G)$ of an undirected graph $G$ is a simplicial complex formed by the sets of vertices in the cliques of $G$.
\end{definition}

\begin{definition}[Link] Let $X$ be a $d$-dimensional simplicial complex and $S\in X(i)$. The link of $S$ is a $(d-i-1)$-dimensional simplicial complex defined by $X_S = \{ T\setminus S : \: S\subset T \in X\}$. When considering the clique-complex of a graph, the link of a clique $C\subset V$ is the clique-complex of the subgraph induced on the joint neighborhood of $C$.
\end{definition}

\begin{definition}[$1$-Skeleton] Let $X$ be a simplicial complex. The $1$-skeleton of $X$ is the graph whose edges are $X(1)$ and vertices are $X(0)$.
\end{definition} 

\begin{definition}[Strongly gallery connected] \label{def:sgc} Let $X$ be a simplicial complex. We say that $X$ is strongly gallery connected if $X$ is connected (i.e., the $1$-skeleton of $X$ is connected) and every link of dimension $\geq 1$ is connected.  
\end{definition}

\subsection{Hyper-regular graphs}

\begin{definition}[hyper-regular graph]
Let $G = (V,E)$ be a graph and let  $\:d_0>d_1>\ldots>d_{n-1}>0$. We say that $G$ is $(d_0, d_1,\ldots, d_{n-1})$-regular if $G$ is $d_0$-regular and for every $0<i\leq n-1$, the link of every $i$-clique is $d_i$-regular.
\end{definition}
Similarly, a simplicial complex $X$ is $(d_0, d_1,\ldots, d_{n-1})$-regular if $X(0)$ is $d_0$ regular and the 1-skeleton of the link of every $(i-1)$-dimensional face is $d_i$-regular. \\
An infinite family of $(d_0, d_1,\ldots, d_{n-1})$ - regular graphs is one that contains these graphs of arbitrarily large size w.r.t the number of vertices. Such families can be constructed trivially from some finite graph by joining multiple copies of it, but the interesting case is when all graphs are strongly gallery connected, and even better- are good expanders.

\begin{observation}
\label{obs:reg}
$G$ is  $(d_0, d_1,\ldots, d_{n-1})$-regular if and only the link of every $i$-clique is \textbf{of size $d_{i-1}$}.
\end{observation}

\subsection{Partite graph product}
\label{sub:partite}
An $n$-partite graph is one where the vertex set is partitioned into $n$-parts, and all edges are between vertices in distinct parts. We will define a product between two such graphs, $H$, $G$, the result being a subgraph of the usual tensor product $H \otimes G$.
\begin{definition}[Partite graph product] 
Let $G_1 = (V_1, E_1), \: G_2 = (V_2, E_2)$ be $n$-partite graphs, where $V_1 = (V_1^1, V_1^2, \ldots, V_1^n), \: V_2 = (V_2^1, V_2^2,\ldots, V_2^n)$ are ordered tuples of mutually disjoint subsets. We define their $n$-partite product, denoted by $\circledast$ , to be
$$
 G_1 \circledast G_2 = G_{12} = (V_{12}, E_{12})
$$
Where
$$
V_{12} = \Big( \: V_1^1 \times V_2^1 \:, \: V_1^2 \times V_2^2 \:, \ldots ,\: V_1^n \times V_2^n \: \Big)
$$
$$
E_{12} =\bigcup_{1\leq i<j\leq n} \Big\{ \big\{(v_1^i, v_2^i), (v_1^j, v_2^j)\big\}: \: (v_1^i, v_1^j) \in E_1 \text{ and } (v_2^i, v_2^j) \in E_2\Big\}
$$
\end{definition}

The transition matrix of the resulting graph is easy to compute. Take for example a product of $3$-partite graphs:
\newline
\newline
$$
\left[
\begin{array}{ccc}
0 & A_{12} & A_{13} \\\\
A_{12}^T & 0 & A_{23} \\\\
A_{13}^T & A_{23}^T & 0
\end{array}
\right]
\bigcircledast
\left[
\begin{array}{ccc}
0 & B_{12} & B_{13} \\\\
B_{12}^T & 0 & B_{23} \\\\
B_{13}^T & B_{23}^T & 0
\end{array}
\right]
=
\left[
\begin{array}{ccc}
0 & A_{12} \otimes B_{12} & A_{13} \otimes B_{13} \\\\
(A_{12}\otimes B_{12})^T & 0 & A_{23} \otimes B_{23} \\\\
(A_{13}\otimes B_{13})^T & (A_{23}\otimes B_{23})^T & 0
\end{array}
\right]
$$
\newline
\begin{remark}
For clique complexes $X(G_1), X(G_2)$, the product is naturally defined as follows:
$$
X(G_1) \circledast X(G_2) = X(G_1 \circledast G_2)  
$$
\end{remark}
We want the product to preserve important properties of the underling graphs, such as connectivity and expansion properties (see \ref{sub:spectral}). It turns out that it is   sufficient for the graphs to be pure and strongly gallery connected.
\begin{claim}
Let $X(G_1), X(G_2)$ be pure n-dimensional, $(n+1)$-partite strongly gallery connected clique complexes.
Then $X(G_1 \circledast G_2)$ is of the same kind, and in particular is strongly gallery connected.
\end{claim}

\begin{proof}
First, we show that $G_{1}\circledast G_{2}$ is connected: Let $(v_i,w_i) \in V_1^i\times V_2^i, (v_j,w_j) \in V_1^j\times V_2^j$.
They are connected if and only if there are paths $v_i, a_{k_1}, a_{k_2},\ldots, a_{k_m}, v_j$ in $G_1$ and $w_i, b_{k_1}, b_{k_2},\ldots, b_{k_m}, w_j$ in $G_2$ of the same length $(m+1)$ such that $a_{k_l} \in V_1^{k_l}$ and $b_{k_l} \in V_2^{k_l}$ for every $l$. Clearly, if either $v_i, v_j$ or $w_i, w_j$ are part of a $d$-dimensional face - i.e., $d+1$-clique with one vertex from each component - then $(v_j,w_j)$ and $(v_i,w_i)$ are connected.\\
$G_1$ is connected, thus there exists a path from $v_i$ to $v_j$. $G_2$ is pure, so $w_i$ is contained in a d+1-clique with some $u_j\in V_2^j$ $\Rightarrow$ there is a path from $(v_i, w_i)$ to $(v_j, u_j)$. $G_2$ is connected, so there exists a path  from $u_j$ to $w_i$. $G_1$ is also pure, therefore $v_j$ in also contained in a $(d+1)$-clique. Hence, there is a path from $(v_j, u_j)$ to $(v_j, w_j)$.\\
A similar proof works for the connectivity of  the links (using the claim below and the fact that a link of a pure clique complex is pure itself). The other properties are trivial.
\end{proof}

\begin{claim}
\label{claim:linkprod}
Let $G_1 = \big((V^i_1)_{i=0}^n,\: E_1\big), G_2= \big((V^i_2)_{i=0}^n,\: E_2\big)$. Let $S \subset V(G_1\circledast G_2)$ be a clique and let $X_S$ be its link. Then, $X_S = X_{S|_{G_1}} \circledast X_{S|_{G_2}}$.
\end{claim}
\begin{proof}
Let $(u, v)\in V_1^i \times V_2^i$:
$$
(u,v)\in X_S \:\iff\: \forall \: (s_1,s_2)\in S, \{(u,v),(s_1,s_2)\} \in E(G_1\circledast G_2) \: \iff  $$
$$
\forall (s_1,s_2)\in S,\:\: \{{u,s_1}\}\in E_1 \:\wedge\: \{{v,s_2}\}\in E_2 \:\iff\: u\in X_{S|_{G_1}} \wedge v\in X_{S|_{G_2}}
$$
And if the vertices sets are equal, by the definition of the product the edges sets are equal as well.
\end{proof}

\subsection{Symmetrization}
\label{sub:sym}
Given a pure, $(n+1)$-partite, $n$-dimensional strongly gallery connected clique complex, we want to use our product to build from it a strongly gallery connected hyper-regular graph. In order for our construction to work, the graph needs to have an additional property.
Let $I= \{0,1,\ldots,n\}$ and let $G= (V= (V^i)_{i\in I}, E)$ be an $(n+1)$-partite clique complex; i.e., for every $i\in I$, $V^i$ is an independent set and for every $j\neq i$, $V^i \cap V^j=\emptyset$. Let $J \subset I$. We say that a clique $C\subset V$ is of type $J$ if $|C| = |J|$ and for every $j\in J$ there is (unique) $v\in C$ such that $v\in V^j$. 

\begin{definition}[Partite type-regular] We say that $G$ is $(d_J)_{J\subsetneq I}$-partite type-regular, where $d_J \in {\bbN}^{I\setminus J}$, if for every $i\in I\setminus J$, the intersection between $V_i$ and the link of each J-clique is of size $d_J(i)$.
\end{definition}

Let $G = (V, E)$ be $(n+1)$-partite graph, with an order on the parts: $V= (V^0, V^1,\ldots, V^n)$, and 
$$E = \bigcup_{0\leq i<j\leq n} E(V^i, V^j)$$ 
Where $E(V^i, V^j)$ denotes the edges between the i-th and the j-th side. Let $S_{n+1}$ be the permutation group on $n+1$ elements.  For $\pi \in S_{n+1}$ define
$$
\pi(G) = \Big((V^{\pi(0)}, V^{\pi{(1)}},\ldots, V^{\pi(n)}),\: \bigcup_{0\leq i<j\leq n} E(V^{\pi{(i)}}, V^{\pi{(j)}}) \Big)
$$
Note that the action merely permutes the labels of the sides; i.e., as graphs $G$ and $\pi(G)$ are the same, but as {\em ordered} multipartite graphs they differ. The order of the sides plays an important role in the $\circledast$ product. We can now describe the symmetrization.
\begin{definition}
Let $G$ be an $(n+1)$-partite graph. Define the symmetrization of $G$ as:
$$
G^{\circledast S_{n+1}} := \mathop{\bigcircledast}_{\pi \in S_{n+1}} \pi(G) 
$$
\end{definition}

\begin{claim}
\label{cla:ptr}
If $G$ is $(n+1)$-partite type-regular, then $G^{\circledast S_{n+1}}$ is hyper-regular.
\end{claim}

\begin{proof}
Let $I=\{0,1,\ldots,n\}$, and let $S_I$ denote the permutation group on the elements in $I$. Let $1\leq m\leq n$, $J,L\subset I$ where $|J|=|L|=m$. The link of every $J$-clique in $G^{\circledast S_{I}}$ is of size
$$
\sum_{i\in I\setminus J} \prod_{\pi\in S_I} d_{\pi^{-1}(J)}(\pi^{-1}(i)) = \sum_{i\in I\setminus J} \prod_{\pi\in S_I} d_{\pi(J)}(\pi(i))
$$
Let $\gamma \in S_I$ such that $\gamma(J)=L$ (and it must be that $\gamma(I\setminus J)=I\setminus L$)
$$
\sum_{i\in I\setminus J} \prod_{\pi\in S_I} d_{\pi(J)}(\pi(i)) = \sum_{i\in I\setminus J} \prod_{\pi\circ \gamma \in S_I\cdot \gamma} d_{\pi\circ \gamma(J)}(\pi\circ \gamma (i)) = \sum_{i\in I\setminus L} \prod_{\pi\in S_I} d_{\pi(L)}(\pi(i))
$$
I.e., for every $1\leq m \leq n$, all cliques of size $m$ have the same link size. Due to observation \ref{obs:reg}, the symmetrization of $G$ is hyper regular.
\end{proof}

\subsection{Spectral expander}
\label{sub:spectral}
Here we prove that the $\circledast$ product on pure, strongly gallery connected partite type-regular graphs preserves their expansion properties in a sense.
\begin{definition}[Transition matrix]
Given a graph $G=(V,E)$ with a probability distribution P on E (where $P(e)\neq 0$ for all $e\in E$), we define the transition matrix $T_G$ on $G$ by
$$
{T_G}{(v,u)} = \Pr(u|v) = \frac{P(\{u,v\})}{\sum_{\{w,v\} \in E} P(\{w,v\})}
$$
\end{definition}

Usually we define a uniform distribution over the edges. When $G$ is $d$-regular, the transition matrix is simply $T_G = \frac{1}{d}A_G$ where $A_G$ is the adjacency matrix. Let $G$ be a bipartite graph. If $G$ is biregular (with degrees $d,k$) its adjacency matrix and transition matrix (w.r.t the uniform distribution over the edges) are of the form
$$A_G= 
\left[
\begin{array}{cc}
0 & B \\[0.4cm]
B^T & 0
\end{array}
\right] \:\:\:\:\:\:
T_G= \left[
\begin{array}{cc}
0 & \frac{1}{k}B \\[0.5cm]
\frac{1}{d}B^T & 0
\end{array}
\right]
$$

For \textbf{hypergraphs}, it's a bit more complicated:\\
Let X be a pure n-dimensional finite simplicial complex. For every $-1\leq k \leq n-2$, $\tau \in X(k)$, the probability distribution of the edges of the one skeleton of $X_{\tau}$ corresponds to the high dimensional structure of X. Explicitly, given $\tau \in X(k)$ and $\{u,v\} \in X_\tau(1)$, the probability of $\{u,v\}$ is proportional to the weight function
$$
w(\{u,v\}) = (n-k-2)!\abs{\{\pi \in X_\tau(n-k-1) :\: \{u,v\}\subset \pi\}}
$$

\begin{note}
For $1$-dimensional links (i.e., $k=n-2$), the weight function coincides with the uniform distribution over the edges.
\end{note}

Let $G$ be a graph and let $A_G, T_G$ be its adjacency and transition matrices, respectively. For a matrix $M$, We denote its i-th largest eigenvalue (including repetitions) by $\lambda_i(M)$, and its spectrum by $\sigma(M)$. In particular, we refer to $\lambda_2(T_G)$ as the second largest eigenvalue of the graph $G$, denoted by $\lambda_2(G)$.

\begin{definition}[Spectral expander]
A graph $G$ is a $\lambda$ - one sided expander if $\lambda \geq \lambda_2(G)$. 
\end{definition}

\begin{definition}\label{hdx}[$\lambda$ - high dimensional expander] Let $\lambda < 1$. A d-dimensional pure simplicial complex $X$ is a $\lambda$ - one sided high dimensional expander if:
\begin{enumerate}
    \item The 1-skeleton of $X$ is a one-sided $\lambda$-spectral expander.
    \item For any $i\leq d-2$ and all $s\in X(i)$, the 1-skeleton of $X_s$ is a one-sided $\lambda$-spectral expander.
\end{enumerate}
\end{definition}
In \cite{opp17}, Oppenheim proves the following:
\begin{lemma}[Trickling-Down Theorem]
\label{lem:trick}
Let $X$ be a $d$-dimensional simplicial complex such that the $1$-skeleton of every link (including the entire simplicial complex) is connected and $\forall s \in X(d-2)$, $X_s$ is a one-sided $\lambda$-expander. Then $X$ is a $\mu$-expander where $\mu = \frac{\lambda}{1 - (d-1)\lambda}$
\end{lemma}

Note that in $r$-partite graphs, every $1$-dimensional link (of a $(r-2)$-clique) is bipartite. It follows from claim \ref{claim:linkprod} that every $1$-dimensional link in a partite graph product is a bipartite product of $1$-dimensional links of the original graphs. Now thanks to the Trickling-Down Theorem, it's enough for us to bound the second eigenvalue of a bipartite product. Useful facts regarding graph spectrum:
\begin{itemize}
    \item For any graph $G$ and transition matrix $T$, all eigenvalues $\lambda$ are $-1\leq \lambda \leq 1$. 
    \item $1$ is always an eigenvalue (of the all-one eigenvector). For any connected graph $G$ (excluding isolated vertices), the largest eigenvalue of the \textbf{adjacency} matrix is unique. 
    \item The spectrum of the adjacency matrix of every bipartite graph is symmetric w.r.t $0$.
\end{itemize}

\begin{proposition}
\label{prop:bip}
Bipartite product of connected biregular graphs preserves the expansion.\\ i.e., for $G_1, G_2$ connected biregular graphs: $$\lambda_2(G_1 \circledast G_2) = \max{\{\lambda_2(G_1), \lambda_2(G_2)\}}$$.
\end{proposition}
\begin{note}
In a partite type-regular graph, in particular, all 1-dimensional links are biregular. 
\end{note}

\begin{proof}
First, we prove this for regular graphs:\\
Let $G_1, G_2$ be connected regular bipartite graphs, with transition matrices
\newline\newline
$$T_1= \left[
\begin{array}{cc}
0 & \frac{1}{d_1}B_1 \\[0.5cm]
\frac{1}{d_1}B_1^T & 0
\end{array}
\right] \:\:\:\: T_2= \left[
\begin{array}{cc}
0 & \frac{1}{d_2}B_2 \\[0.5cm]
\frac{1}{d_2}B_2^T & 0
\end{array}
\right]
$$
\newline\newline
Let 
$$\frac{1}{d_1}B_1\:=\: U_1 \frac{1}{d_1} \Sigma_1 V_1^T; \:\:\:\: \frac{1}{d_2}B_2\:=\: U_2 \frac{1}{d_2}\Sigma_2 V_2^T$$ be their singular value decomposition (henceforth SVD). Now,
\begin{enumerate}
    \item One can verify that the eigenvalue decomposition of $G_i$ is given by
    \newline\newline
    $$T_i= \left[
    \begin{array}{cc}
    \hat{U}_i & \hat{U}_i \\[0.5cm]
    \hat{V}_i & -\hat{V}_i
    \end{array}
    \right]
    \left[
    \begin{array}{cc}
    {\frac{1}{d_i}\Sigma_i} & {0} \\[0.5cm]
    {0} & -{\frac{1}{d_i}\Sigma_i}
    \end{array}
    \right]
    \left[
    \begin{array}{cc}
    \hat{U}_i & \hat{U}_i \\[0.5cm]
    \hat{V}_i & -\hat{V}_i
    \end{array}
    \right]^{T}
    $$
    \newline\newline
    Where $\hat{U}_i=U_i/\sqrt{2}, \hat{V}_i=V_i/\sqrt{2}$ (see section 2.2 in \cite{kun}). I.e., the spectrum of $G_i$ is $ \pm \sigma(\frac{1}{d_i}\Sigma_i)$. 
    \item The SVD of $\frac{1}{d_1}B_1\otimes \frac{1}{d_2}B_2$ is $(U_1 \otimes U_2) (\frac{1}{d_1}\Sigma_1 \otimes \frac{1}{d_2}\Sigma_2) (V_1^T \otimes V_2^T)$ (Kronecker product properties).\\
    This implies that the spectrum of $G_1\circledast G_2$, i.e., of the transition matrix
    \newline\newline
    $$
    T_1 \circledast T_2 = 
    \left[
    \begin{array}{cc}
    0 & \frac{1}{d_1 d_2}B_1\otimes B_2 \\[0.7cm]
    (\frac{1}{d_1 d_2}B_1 \otimes B_2)^T & 0
    \end{array}
    \right]
    $$
    \newline\newline
    is $\pm \sigma(\frac{1}{d_1 d_2}\Sigma_1 \otimes \Sigma_2)$
\end{enumerate}
Note that $\sigma(A_{G_1\circledast G_2}) = \pm \sigma(\Sigma_1 \otimes \Sigma_2)$, i.e. it is the spectrum of the \textbf{adjacency} matrix of $G_1\circledast G_2$.
Since bipartite product of connected bipartite graphs results in a connected graph (following the proof of claim 2.1, every connected bipartite graph is pure), the multiplicity of the largest eigenvalue is one. Since the spectrums are symmetric, the two largest eigenvalues in each adjacency matrix are also the largest in absolute value, so  
$$
\lambda_2(A_{G_1\circledast G_2}) \leq \max\big\{\lambda_1(A_{G_1})\cdot \lambda_2(A_{G_2}), \: \lambda_2(A_{G_1})\cdot \lambda_1(A_{G_2})\big\}
$$
thus,
$$
\lambda_2(G_1\circledast G_2) \leq \max\Big\{\frac{\lambda_1(A_{G_1})}{d_1}\cdot \frac{\lambda_2(A_{G_2})}{d_2}, \: \frac{\lambda_2(A_{G_1})}{d_1}\cdot \frac{\lambda_1(A_{G_2})}{d_2}\Big\} = \max\{\lambda_2(G_2), \lambda_2(G_1)\}
$$
Now, we extend the proof for biregular graphs.

\begin{claim}
Let $G$ be a biregular graph with left side of size $n$ and degree $d$, and right side of size $m$ and degree $k$, whose adjacency matrix and transition matrix are
$$
A_G = \left[
\begin{array}{cc}
0 & B \\[0.4cm]
B^T & 0
\end{array}
\right]
\:\: 
T_G = \left[
\begin{array}{cc}
0 & \frac{1}{k}B \\[0.5cm]
\frac{1}{d}B^T & 0
\end{array}
\right]
$$
It holds that $\sigma (T_G) = \frac{1}{\sqrt{d k}} \cdot \sigma(A_G)$
\end{claim}

\begin{proof}
Let $\lambda \in \sigma(A)$, and let $v = \begin{pmatrix} v_1\\ \hline v_2\end{pmatrix}$ be its eigenvector, where $v_1 \in \bbR^n, v_2 \in \bbR^m$. \\
So, $T \cdot \begin{pmatrix} \sqrt{d} \cdot v_1\\ \hline \sqrt{k}\cdot  v_2\end{pmatrix} =
\begin{pmatrix} \sqrt{k}\cdot \frac{\lambda}{k} v_1\\ \hline \sqrt{d} \cdot  \frac{\lambda}{d} v_2\end{pmatrix} \:=\:
\frac{1}{\sqrt{d k}}\lambda \begin{pmatrix} \sqrt{d}\cdot  v_1\\ \hline \sqrt{k} \cdot  v_2\end{pmatrix}$.\\
And it works similarly in the other direction. In addition, every $\lambda \in \sigma(A_G)$ has the same multiplicity as $\frac{\lambda}{\sqrt{d k}} \in \sigma(T_G) $: It's easy to verify that $\bigg\{\begin{pmatrix} v_1^1\\ \hline v_2^1\end{pmatrix}, \begin{pmatrix} v_1^2\\ \hline v_2^2\end{pmatrix}, \ldots, \begin{pmatrix} v_1^t\\ \hline v_2^t\end{pmatrix}\bigg\}$ is independent if and only if $\bigg\{\begin{pmatrix} \sqrt{d}\cdot v_1^1\\ \hline \sqrt{k}\cdot v_2^1\end{pmatrix}, \begin{pmatrix} \sqrt{d}\cdot v_1^2\\ \hline \sqrt{k}\cdot v_2^2\end{pmatrix}, \ldots, \begin{pmatrix} \sqrt{d}\cdot v_1^t\\ \hline \sqrt{k}\cdot v_2^t\end{pmatrix}\bigg\}$ is.  
\end{proof}

\remove{
Note that biregular graphs with different degrees don't have the same number of vertices in each side. We need equal size sides in order to use the result above, so we pad the smaller side with isolated vertices.

\begin{claim}
Let $G$ be a graph. Adding an isolated vertex to $G$ doesn't change its spectrum besides another 0 as eigenvalue.
\end{claim}

\begin{proof}
Adding an isolated vertex to $G$ means adding an all-zero column and row, both at the same place, $i$. Its easy to verify that if $\lambda$, $v$ were eigenvalue - eigenvector pair before, then $\lambda, v'$ are such pair after the padding, where $v'$ is $v$ padded with 0 between the $i-1$ and the $i^{th}$ entry. The extra eigenvector is $e_i$ (with eigenvalue 0).
\end{proof}
}
In conclusion, Let $G_1, G_2$ be biregular graphs with degrees $d, k$ and $p,q$. in the above notation, their spectrums are $\pm \sigma (\frac{1}{\sqrt{d k}}\Sigma_1), \: \pm\sigma (\frac{1}{\sqrt{p q}}\Sigma_2)$ and their product's spectrum is, as expected,  $\pm \sigma (\frac{1}{\sqrt{d k p q}}\Sigma_1\otimes \Sigma_2)$, and the product is again connected (excluding isolated vertices), thus
$$
\lambda_2(G_1\circledast G_2) \leq \max\Big\{\frac{\lambda_1(A_{G_1})}{\sqrt{d k}}\cdot \frac{\lambda_2(A_{G_2})}{\sqrt{p q}}, \: \frac{\lambda_2(A_{G_1})}{\sqrt{d k}}\cdot \frac{\lambda_1(A_{G_2})}{\sqrt{p q}}\Big\} = \max\{\lambda_2(G_2), \lambda_2(G_1)\}
$$
\end{proof}

\begin{remark}
Note that the above isn't true for $n$-partite products with $n > 2$: In particular, when $n>2$, the product of connected graphs is not necessarily connected.
See Question \ref{sum:1} at the end of this paper.

\end{remark}

\begin{corollary}
If $G_1, G_2$ are both pure d-dimensional, (d+1)-partite strongly gallery connected  clique complexes where all 1-dimensional links are biregular, $\mu_1$ and respectively $\mu_2$
-expanders, where $\mu_1, \mu_2$ are consistent with the Trickling-Down Theorem, then $G_1 \circledast G_2$ is a $\max\{\mu_1, \mu_2\}$-expander.
\end{corollary}

\section{Subgroup geometry system}
\label{sec:sgs}
So far, we have seen that pure, strongly gallery connected partite type-regular graphs can be fully regularized by our symmetrization while preserving strong gallery connectivity and some expansion property of the original graph. In this section, we show that clique complexes arising from subgroup geometry systems are pure, strongly gallery connected partite type-regular, and thus can be regularized.
We state here relevant definitions and theorems from Kaufman and Oppenheim's work. For further reading and deeper understanding, see Kaufman and Oppenheim \cite{kaufman2017simplicial} and chapter 1 in Diagram Geometry \cite{BC13}, about incidence geometry.

\begin{definition}[Coset complex]
Given a group $G$ with subgroups $K_{\{i\}}$, $i\in I$, where $I$ is finite, the coset complex $X= X\big(G,(K_{\{i\}}
)_{i\in I}\big)$ is a simplicial complex defined as follows:
\begin{enumerate}
    \item The vertex set of $X$ is the set of cosets of the subgroups, i.e., $X(0) = \bigcup_{i\in I}\{g K_{\{i\}} : g\in G\}$.
    \item Two vertices $g K_{\{i\}}, g' K_{\{j\}}$ form an edge, i.e., $\{g K_{\{i\}}, g' K_{\{j\}}\} \in X(1)$ if $i\neq j$ and $g K_{\{i\}} \cap g' K_{\{j\}} \neq \emptyset$.
    \item The simplicial complex $X$ is the clique complex spanned by the $1$-skeleton defined above, i.e., $\{g_0 K_{\{i_0\}}, \ldots, g_k K_{\{i_k\}}\} \in X(k)$ if for every $0\leq j, j'\leq k$, $\:\: g_{j} K_{\{i_j\}} \cap g_{j'} K_{\{i_{j'}\}} \neq \emptyset$.
\end{enumerate}
\end{definition}
We denote for every $\tau \subset I, \tau\neq \emptyset$, $K_{\tau} := \bigcap_{i\in \tau} K_{\{i\}}$, and further denote $K_{\emptyset} := G$.

\begin{definition}[Subgroup geometry system]
\label{def:sgs}
Let $n\in \bbN, I$ finite set of cardinality $n+1$ and $G$ be a group with subgroups $K_{\{i\}}$, $i\in I$. We call $\big( G, (K_{\{i\}})_{i\in I}\big)$ a subgroup geometry system if:
\begin{itemize}
    \item (A1) For every $\tau, \tau' \subset I$, $K_{\tau\cap \tau'} = \langle  K_{\tau}, K_{\tau'}\rangle$, where $\langle  K_{\tau}, K_{\tau'}\rangle$ denotes the subgroup of $G$ generated by $K_{\tau}, K_{\tau'}$.
    \item (A2) For every $\tau \subsetneq I$ and $i\in I\setminus \tau$, $K_{\tau} K_{\{i\}} = \bigcap_{j\in \tau} K_{\{j\}} K_{\{i\}}$.
    \item (A3) For every $i\in I$, $K_{I} \neq K_{I\setminus \{i\}}$.
\end{itemize}
\end{definition}
We'll need the following lemmas, stated and proved in \cite{kaufman2017simplicial}:
\begin{lemma}
Let $n \in \bbN, I$ a finite set of cardinality $n+1$, and $G$ be group with subgroups $K_{\{i\}}, i \in I$.
If $\big(G,(K_{\{i\}}
)_{i\in I}
\big)$ is a subgroup geometry system, then
\begin{enumerate}
\item $X = X\big(G,(K_{\{i\}}
)_{i\in I}
\big)$
is a pure $n$-dimensional, $(n + 1)$-partite, strongly gallery connected clique complex. Furthermore, for every $\sigma, \sigma' \in X(n)$ there is $g \in G$ such that $g.\sigma = \sigma'$.
\item As a consequence, for $\tau \subsetneq I$, every clique of type $\tau$ is of the form $g K_{\tau}$ for some $g\in G$. More precisely, for every clique  $\{g_0 K_{\{i_0\}}, \ldots, g_k K_{\{i_k\}}\} \in X(k)$ there exists $g\in G$ such that  $\{g_0 K_{\{i_0\}}, \ldots, g_k K_{\{i_k\}}\} = \{g K_{\{i_0\}}, \ldots, g K_{\{i_k\}}\}$ 
\item Given $g \in G, \tau \subsetneq I$, the link of $g K_{\tau}$ in $X\big(G,(K_{\{i\}}
)_{i\in I}\big)$ is isomorphic to $X\big(K_{\tau} ,(K_{\tau \cup \{i\}}
)_{i\in I\setminus \tau}\big) $.
\end{enumerate}
\end{lemma}

\begin{theorem}
Every subgroup geometry system is partite type-regular.
\end{theorem}

\begin{proof}
Let $\tau \in I$. We know that every clique of type $\tau$ is of the form $g K_{\tau}$. It's easy to check that $G \xhookrightarrow{} Aut(X)$: The action of $G$ on the vertices of $X$, defined by $g.(g' K_{\{i\}}) = g g' K_{i}$ for $g, g' \in G, i\in I$ dictates a type preserving automorphism of $X$ which is simplicial, i.e., sends type $\sigma$ simplex to a type $\sigma$ simplex. So, for every $g_1 K_{\tau}, g_2 K_{\tau}$ it follows that their links are isomorphic (by the automorphisms $g_1 {g_2}^{-1}, g_2 g_1^{-1}$), and since these automorphisms are type preserving, for every $i\in I\setminus \tau$, the i'th sides of the links are of the same size.
\end{proof}

\begin{corollary}
For subgroup geometry system $\big(G,(K_{\{i\}})_{i\in I}\big)$, the coset complex 
$$ X^{\circledast S_{I}}:=\mathop{\bigcircledast}_{\pi \in S_I} \pi(X(G,(K_{\{i\}})_{i\in I})) $$ 
is strongly gallery connected HRG. In addition, if every $1$-dimensional link of $X$ is $\lambda$-expander, then $ X^{\circledast S_{I}}$ is a $\frac{\lambda}{1-(n-1)\lambda}$-expander. 
\end{corollary}

\subsection{Another proof of hyper-regularity}
We present an additional approach, as this angle will be needed in the following sections.

\begin{claim}
\label{cla:tran}
Let $G$ be a group with subgroups $K_{\{i\}}$, $i\in I$, and let $X= X\big(G,(K_{\{i\}}
)_{i\in I}\big)$ be the corresponding coset complex. If for every $1\leq m \leq n$, every two intersections of $m$ different subgroups from $K_{\{i\}}$, $i\in I$ are isomorphic (enough to be of the same size actually), then $X$ is hyper-regular. 
\end{claim}

\begin{proof}
Let $1\leq m \leq n$. We show that the links of any two m-cliques are of the same size: Let $J,L\subset I$, both of size $m$. We saw that the links of any cliques of type $J, L$ are isomorphic to
$$
X(K_{J} ,(K_{J \cup \{i\}}
)_{i\in I\setminus J} 
), \:\: X(K_{L} ,(K_{L \cup \{i\}}
)_{i\in I\setminus L} 
)
$$
And since every equal-size intersections are isomorphic, their sizes are
$$
\sum_{i\in I\setminus J} [{K_J}\: : \: K_{J\cup\{i\}}] =  (n+1 - m) [{K_J}\: : \:K_{J\cup\{i\}}] =  (n+1 - m) [{K_L}\: : \:K_{L\cup\{i'\}}] = \sum_{i'\in I\setminus L} [{K_L} \: :\: K_{L\cup\{i'\}}]
$$
\end{proof}

\begin{note}
$\pi_1(X) \circledast \pi_2(X) \cong X(G\times G, (K_{\{\pi_1(i)\}}\times K_{\{\pi_2(i)\}}
)_{i\in I})$
\end{note}

\begin{claim}
\label{cla:closure}
Let $I\subset \bbN$ finite, $G_1, G_2$ be groups and $K_{1_i} \leq G_1,\: K_{2_i} \leq G_2$ for every $i\in I$. \\If $(G_1, (K_{1_i})_{i\in I})$ and $(G_2, (K_{2_i})_{i\in I})$ are subgroup geometry systems then so is 
$$(G_1 \times G_2, (K_{1_i}\times K_{2_i})_{i\in I})$$
\end{claim}

\begin{proof}
Easy to verify (A1), (A2) and (A3).
\end{proof}

\begin{corollary}
If $X(G,(K_{\{i\}}
)_{i\in I})$ is a subgroup geometry system, then so is  
$$X^{\circledast S_{I}} = \mathop{\bigcircledast}_{\pi \in S_{I}} \pi(X) \cong X\Big(G^{(n+1)!},(\prod\limits_{\pi \in S_{I}}K_{\{\pi(i)\}}
)_{i\in I}\Big)$$
\end{corollary}

\begin{claim}
$X^{\circledast S_{I}}$ is hyper-regular.
\end{claim}

\begin{proof}
Let $L, J \subset I$ of the same size. Let $\gamma \in S_{I}$ such that $\gamma(L)=J$. Then:
\vspace{0.5cm}
$$
\bigcap\limits_{j\in J} \prod\limits_{\pi \in S_{I}} K_{\{\pi(j)\}} = \prod\limits_{\pi \in S_{I}} K_{\pi(J)} =\prod\limits_{\pi \in S_{I}} K_{\pi \circ \gamma(L)} \cong \prod\limits_{\pi' \in S_{I}\cdot \gamma^{-1}} K_{\pi' \circ \gamma(L)} = \prod\limits_{\pi \in S_{I}} K_{\pi(L)} = \bigcap\limits_{l\in L}  \prod\limits_{\pi \in S_{I}} K_{\{\pi(l)\}}
$$
\end{proof}

\section{Construction method}
\label{sec:method}
Our constructions rely on graphs that arise from a subgroup geometry systems. We want to generate, for every $n$, an infinite family of finite $n$-dimensional strongly gallery connected HRG HDX.\\ 
We will say that a complex is a  $\lambda$-local-spectral expander if its 1-dimensional skeleton is a $\lambda$- one sided expander.
Formally, our objective is to construct, given $0<\lambda<1$, $n>1$, a family of pure $n$-dimensional finite hyper-regular simplicial complexes $\{X^{(s)}\}_{s\in \bbA}$, where $\bbA\in \bbN$ is an infinite set such that the following holds:
\begin{enumerate}
    \item There are $0<d_{n-1}< \ldots <d_0$ such that For every $s\in \bbA$, $X^{(s)}$ is $(d_0, \ldots, d_{n-1})$-regular. 
    \item For every $s\in \bbA$, $X^{(s)}$ is a one sided $\lambda$-local-spectral expander. 
    \item The number of vertices tends to $\infty$ with $s$.
\end{enumerate}

Given an $n$-dimensional infinite, locally finite subgroup geometry system, we do so by taking appropriate finite (arbitrarily large) quotients of the underlying group and subgroups which yield a finite subgroup geometry system that has the same link structure as the infinite graph. Then, we apply the symmetrization on the finite system, and get a finite (arbitrarily large) strongly gallery connected  $n$-dimensional HRG that has the expansion properties of the infinite graph.
\\\\
Generally, one can use the following lemma, proved by Kaufman and Oppenheim (Proposition 2.12 in \cite{kaufman2017simplicial}):
\begin{lemma}
\label{lem:quo}
Let $(G,(K_{\{i\}})_{i\in I})$ be a subgroup geometry system and let $N \triangleleft G$ be a normal subgroup. Assume that for every $i\in I$, $K_{\{i\}} \cap N = \{e\}$. Then $(G/N,(K_{\{i\}} N/N)_{i\in I})$ is a subgroup geometry system. Furthermore, if we denote $X_G = X(G,(K_{\{i\}})_{i\in I})$ and $X_{G/N}= X(G/N,(K_{\{i\}} N/N)_{i\in I})$, then the following holds:
\begin{enumerate}
    \item $X_G$ and $X_{G/N}$ have the same links. Specifically: for every non-empty $\tau \subseteq I$ and every $g \in G$ the link of $(gN)(K_\tau N)/N$ in $X_{G/N}$ is isomorphic to the link of $gK_\tau$ in $X_G.$

    \item If $G/N$ is a finite group, then $X_{G/N}$ is a finite complex and $|X_{G/N}(n)|\leq |G/N|$.
    \item $X_G$ is a covering of $X_{G/N}$.
\end{enumerate}
\end{lemma}

\section{Constructions}
\label{sec:constructions}
Here we state two main examples of subgroup geometry systems, where one of them has the desired expansion properties in its $1$-dimensional links. 

\subsection{Elementary matrices groups}
\label{sec:elem}
 Recently, Kaufman and Oppenheim discovered new examples of groups with subgroup geometry systems, using elementary matrices groups and Steinberg groups. We state here a short description and their results. See sections 3, 4 in \cite{kaufman2017simplicial} for full elaboration, and Dinur's notes about the 2-dimensional case \cite{HDX}.\\\\ 
 Let $R$ be a unital commutative ring and $\cR$ be finitely generated $R$-algebra. Let $\{t_1,\ldots,t_{l}\}$ be a generating set of $\cR$ and let $T$ be the $R$-module generated by $\{1,t_1,\ldots,t_l\}$. i.e., $T=\{r_0+ \sum_i r_i t_i : r_i \in R\}$.  For $0 \leq i, j \leq n$, $i\neq j$ and $r \in \cR$, let $e_{i,j} (r)$ be the $(n + 1) \times (n + 1)$ matrix with 1?s along the main diagonal, $r$ in the $(i, j)$ entry and $0$?s in all the other entries. The group of elementary matrices denoted $EL_{n+1}(\cR)$ is the group generated by the elementary  matrices with coefficients in $\cR$, i.e., $EL_{n+1}(\cR)= \langle e_{i,j}(r): 0\leq i,j\leq n, i\neq j, r\in \cR \rangle$.\\
 Let $R, \cR, \{1, t_1,\ldots,t_l\}$ and $T$ as above, $n\geq 2$ and $I=\{0,\ldots,n\}$. For $i\in I$ define $K_{\{i\}}< EL_{n+1}(\cR)$ by
 $$
 K_{\{i\}} = \langle e_{j,j+1}(m): j\in I\setminus \{i\}, m\in T \rangle
 $$
 Where $j+1$ is taken modulo $n+1$. Now,
$$( EL_{n+1}(\cR), (K_{\{i\}})_{i\in I})$$
is a subgroup geometry system (see theorem 3.5  in \cite{kaufman2017simplicial}). \paragraph{Infinite family construction} Let $n\geq 2$ and let $q> (n-1)^2$ be a prime power. For $s\in \bbN$, let $\bbF_{q}[t]/\langle t^s\rangle$ be the $\bbF_q$ algebra with the generating set $\{1,t\}$. Let $X^{(s)}$ be the simplicial complex of the subgroup geometry system of $EL_{n+1}(\bbF_{q}[t]/\langle t^s\rangle)$ described above. Then for every $s>n$, the following holds (Theorem 4.10 in \cite{kaufman2017simplicial}):
\begin{enumerate}
    \item $X^{(s)}$ is a pure $n$-dimensional, $(n+1)$-partite, strongly gallery connected clique complex.
    \item $X^{(s)}$ is finite and the number of vertices of $X^{(s)}$ tends to infinity as s tends to infinity.
    \item Each 1-dimensional link of $X^{(s)}$ has a spectral gap at most $\frac{1}{\sqrt{q}}$.
\end{enumerate}
Thus, given any $\lambda > 0$, for large enough $q$: 
$$\{\bigcircledast_{\pi \in S_I} \pi\big(X^{(s)}\big) : s>2^{n-1}\}$$ 
is an infinite family of n-dimensional, hyper-regular, $\lambda$-local spectral expanders.

\subsection{Type $\tilde{A}_{n-1}$ Coxeter group}
\label{sec:triang}
{\bf Overview:} In this section we present another family of geometry systems each giving rise to an infinite family of type-regular graphs, which after regularization via symmetrization yield HRG of high dimension. The group we use for this construction is a type $\tilde{A}_{n-1}$ Coxeter group.
This group has a realization as a group of symmetries of triangulations of the Euclidean space. This has several  qualitative and pedagogical consequences; The geometric interpretation, and especially its 2-dimensional case, are relatively easy to visualize and understand, as the 2-dimensional case simply gives rise to the $(6,2)$-regular graphs that are the hexagonal tilings of finite tori. A small variation of this, example \ref{subsub:3r} in subsection \ref{sub:ad}, yields $(2 \binom{3 r}{r}, \binom{2r}{r})$-regular graphs for every integer $r>1$, which have a simple combinatorial description. 
An unfortunate consequence of the Euclidean origin of these graphs is that they have a natural periodic embedding in Euclidean space, which rules out the possibility of them being good expanders, due to the isoperimetric inequality. This did, however, lead us to discovering the $(120,12,5,2)$-regular example that is embedded in hyperbolic space. 

\subsubsection{Affine permutation group}
The affine permutation group is a well known type $\tilde{A}_{n-1}$ Coxeter group (see section 8.3 in \cite{coxbook}).
\begin{definition}
\label{def:afper}
Let $\Tilde{S}_{n}$ be the group of affine permutations of the integers. i.e., the group of all permutations $u$ of the set $\bbZ$ such that
$$
u(j+n) = u(j) + n \:\: \forall j\in \bbZ
$$
and
$$
\sum_{i=1}^{n} u(i) = \binom{n+1}{2}
$$
\end{definition}
Clearly, such a $u$ is uniquely determined by its values on $\{1,\ldots,n\}$, and we write $u = [a_1,\ldots,a_{n}]$ to mean that
$u(i) = a_i$. As a set of generators for $\Tilde{S}_{n}$, we take $\{\Tilde{s}_0,\ldots, \tilde{s}_{n-1}\}$ where
$$
\Tilde{s}_i = [1,\ldots,i-1,i+1,i,i+2,\ldots,n]
$$
for $0< i\leq n-1$ and
$$
\Tilde{s}_{0} = [0,2\ldots,n-1, n+1]
$$
Define, naturally, the subgroups $K_{\{i\}}\leq \tilde{S}_{n}$ where $K_{\{i\}} =\langle\{\tilde{s}_j: j\neq i\} \rangle$ for $0\leq i \leq n-1$.\\
\begin{proposition}
\label{lem:affine}
 Let $I = \{0,1,\ldots,n-1\}$.
\begin{enumerate}
    \item  $(\tilde{S}_{n}, (K_{\{i\}})_{i\in J})$ is a subgroup geometry system for every $J \subset I$, $|J|>1$.
    \item Every set of $n-1$ generators form an isomorphic copy of the permutation group on $n$ elements, i.e., $K_{\{i\}} \cong S_{n}$ according to the generators map $\tilde{s_i}\rightarrow (i, i+1 \mod n)$.
    \item Let $J,L\subseteq I$. Then $\langle\{\tilde{s}_i: i\in J\}\rangle \cap \langle\{\tilde{s}_i: i\in L\}\rangle = \langle\{\tilde{s}_i: i\in J\cap L\}\rangle$. \\E.g., $K_{\{j, l\}} := K_{\{j\}}\cap K_{\{l\}} = \langle\{\tilde{s}_i: i \notin \{j,l\}\} \rangle$.
\end{enumerate}
\end{proposition}

\begin{proof}
1 and 3 follows from theorem 5.2 in \cite{cohen}. 2 is easy to verify.
\end{proof}

\begin{corollary}
\label{cor:permsize}
For $\tau \subset I, \tau\neq \emptyset$, $K_{\tau} = \langle\{\tilde{s}_i: i \notin \tau\} \rangle \cong \langle\{(i, i+1 \mod n): i \notin \tau \}\rangle \leq S_{n}$
\end{corollary}

\paragraph{Expansion} 
\label{par:afexp}
From \ref{cor:permsize} it follows that for $\tau\subset I, |\tau| = n-2$,  the subgroup $K_\tau$ is isomorphic to either $S_2 \times S_2$ or $S_3$, thus all $1$-dimensional links are isomorphic to a cycle, either $C_4$ or $C_6$. It is known that $\lambda_2(C_4)= 0, \lambda_2(C_6)= 0.5$, hence the upper bound on the expansion parameter we can obtain using \ref{prop:bip}, \ref{lem:trick} is $\frac{0.5}{1-(n-2)0.5}$, which is unfortunately trivial for $n>2$.  

\paragraph{Infinite family construction} The group $\tilde{S}_{n}$ acts on $\bbZ$. We restrict the action of $\tilde{S}_{n}$ to appropriate finite sets.
\begin{proposition}
 For $k\geq 1$, let ${\tilde{S}_{n}}(k)$ denote the group of affine permutations of $\bbZ_{k\cdot n}$, defined similarly, where the second condition in the definition is taken $\mod k\cdot n$. Then, the following holds:
\begin{enumerate}
    \item ${\tilde{S}_{n}}(k)$ is well defined.
    \item ${\tilde{S}_{n}}(k)$ is isomorphic to a quotient group of ${\tilde{S}_{n}}$, i.e., the following mapping is an homomorphism:
    $$
    \varphi: {\tilde{S}_{n}} \rightarrow {\tilde{S}_{n}}(k)
    $$
    $$
    \varphi([a_1,\ldots,a_{n}]) = [a_1 \mod k\cdot n,\:\ldots\:, a_{n} \mod k\cdot n]
    $$
    \item Let $N= Ker(\varphi)$. For every $i\in \{0,1,\ldots,n-1\}$, $N\cap K_{\{i\}} = \{e\}$
    \item $|{\tilde{S}_{n}}(k)| = k^{n-1}n!$
\end{enumerate}
\end{proposition} 
\begin{proof}
Easy to verify.
\end{proof}
\begin{corollary}
$({\tilde{S}_{n}}/N,(K_{\{i\}} N/N)_{i\in I})$ is a finite subgroup geometry system and $X_{\tilde{S}_{n}(k)} \cong X_{{\tilde{S}_{n}}/N}$ has the same links as $X_{\tilde{S}_{n}}$. Thus, 
$$\{\bigcircledast_{\pi \in S_I} \pi\big(X_{{\tilde{S}_{n}}(k)}\big) : k\in \bbN\}$$ 
is an infinite family of strongly gallery connected $n$-dimensional HRG.
\end{corollary}
\begin{proof}
Lemma \ref{lem:quo}.
\end{proof}

\subsubsection{Group of symmetries of a triangulation of $\bbR^{n-1}$}
In this subsection we describe a realization of an affine Coxeter group of type  $\tilde{A}_{n}$ as a group of symmetries. Although the resulting group is the same, this is not the traditional set of symmetries usually used to generate the group.
\begin{figure}
\centering
\subfigure[Labeling simplices and coloring vertices, where $(0,1,2)=(blue, green, yellow)$.]{\includegraphics[scale=0.72]{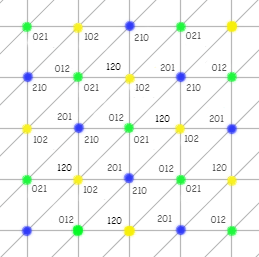}}
\hfill
\subfigure[The cosets of $F_2$; The arrows are the action of $f_1$.]{\includegraphics[scale=0.7]{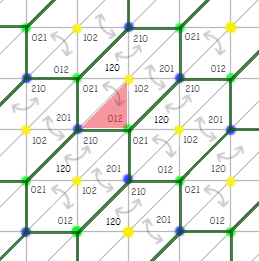}}
\label{fig:clrb}
\hfill
\subfigure[The resulting graph for $n=3$: $(6,2)$-regular.]{\includegraphics[scale=0.35]{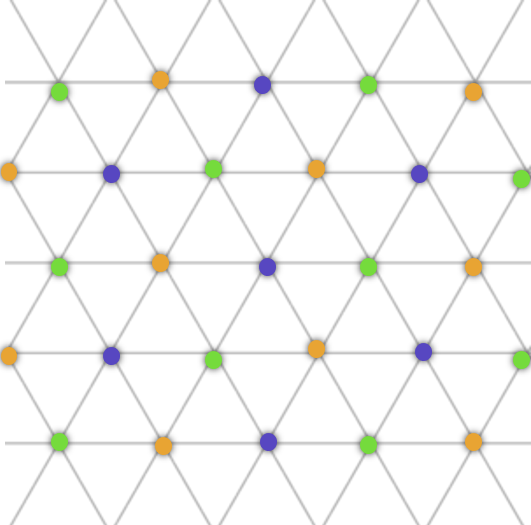}}
\caption{Demonstration on $\bbR^2$}
\label{fig:clr}
\vspace{1cm}
\end{figure}

Let us first specify the set that our group acts on. For $n> 2$:
\begin{enumerate}
\item Triangulate $[0,1]^{n-1}$ by the simplices 
$\Delta_\sigma = \{x: x_{\sigma(1)}\geq x_{\sigma(2)}\geq \ldots\geq x_{\sigma(n-1)} \}$ for every $\sigma$ in the permutation group $S_{n-1}$ (i.e., each simplex corresponds to an increasing path on the cube structure from $0^{n-1}$ to $1^{n-1}$), and triangulate the whole space by the integral shifts of this triangulation. 
\item Coloring each vertex $x=(x_1, x_2, \ldots, x_{n-1})\in \bbZ^n$ by $\sum_{i}x_i \mod{n}$, gives a proper coloring of the graph consisting of the 1-skeleton of this triangulation.
\item Each simplex is now naturally labeled by an element of $S_{n}$: Consider first the cube $[0,1]^{n-1}$. The label of a simplex corresponding to $\sigma\in S_{n-1}$ is $(0, \sigma(1), \sigma(2),...,\sigma(n-1))$. The label of a simplex with the same orientation, in a cube shifted by a vector $v\in \bbZ^{n-1}$ where $\sum_iv_i = c$ is $$(0+ c \mod n,\: \sigma(1)+ c\mod n,\: \sigma(2)+c \mod n,...,\:\sigma(n-1) + c \mod n)$$

\begin{figure}[H]
\centering
\subfigure[Triangulated square]{\includegraphics[scale=0.15]{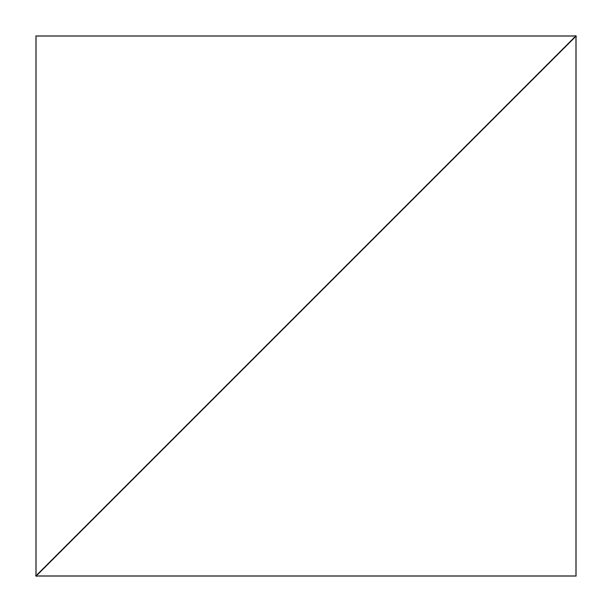}}
\hspace{3cm}
\subfigure[Triangulated cube]{\includegraphics[scale=0.3]{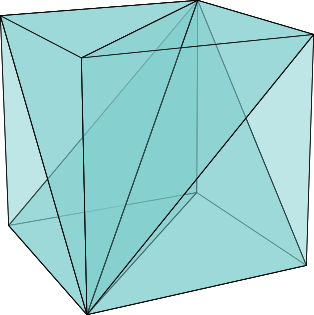}}
\end{figure}

\end{enumerate}

Let $f_i$ be the involution that swaps between simplices sharing a face not containing the vertex colored by $i$. $f_i$ is a mapping from the set of simplices to itself.\\
Define $F_i := \langle\{f_0, f_1,\ldots,f_{n-1}\}\setminus \{f_i\}\rangle$, and Let $G_{n}=  \langle\{f_0, f_1,\ldots,f_{n-1}\}\rangle$.

\begin{claim}
\begin{enumerate}
\item $f_i$ swaps between
neighbouring simplices labeled by $\sigma$ and $(i, i+1)\sigma $. 
\item $F_i\cong S_n$ for every $i$.
\item $G_{n}=  \langle\{f_0, f_1,\ldots,f_{n-1}\}\rangle$ acts simply transitively on the simplices, thus can be identified by the set of simplices where we fix a $(0,...,n-1)$-labeled simplex to be the identity (say, the red colored in figure \ref{fig:clr}.b).
\end{enumerate}
\end{claim}
\begin{proof}
Follows from the next paragraph, where we show that this group is essentially the affine permutation group, and as a Coxeter group it acts simply transitively.
\end{proof}
As a consequence, we can identify each coset with the vertex that in its center. Meaning, $G$'s vertex set can be seen as $\bbZ^n$. For example, in figure \ref{fig:clr}.b, the cosets of $F_2$ are identified by the colored 2 (yellow) vertices. Furthermore, pairs of vertices that are neighbors in the 1-skeleton of the triangulation correspond to cosets that intersect, and $i$-cliques correspond to $i$-intersection of cosets.
\paragraph{Connection to the affine permutation group}
\begin{figure}
\centering
\includegraphics[scale=0.25]{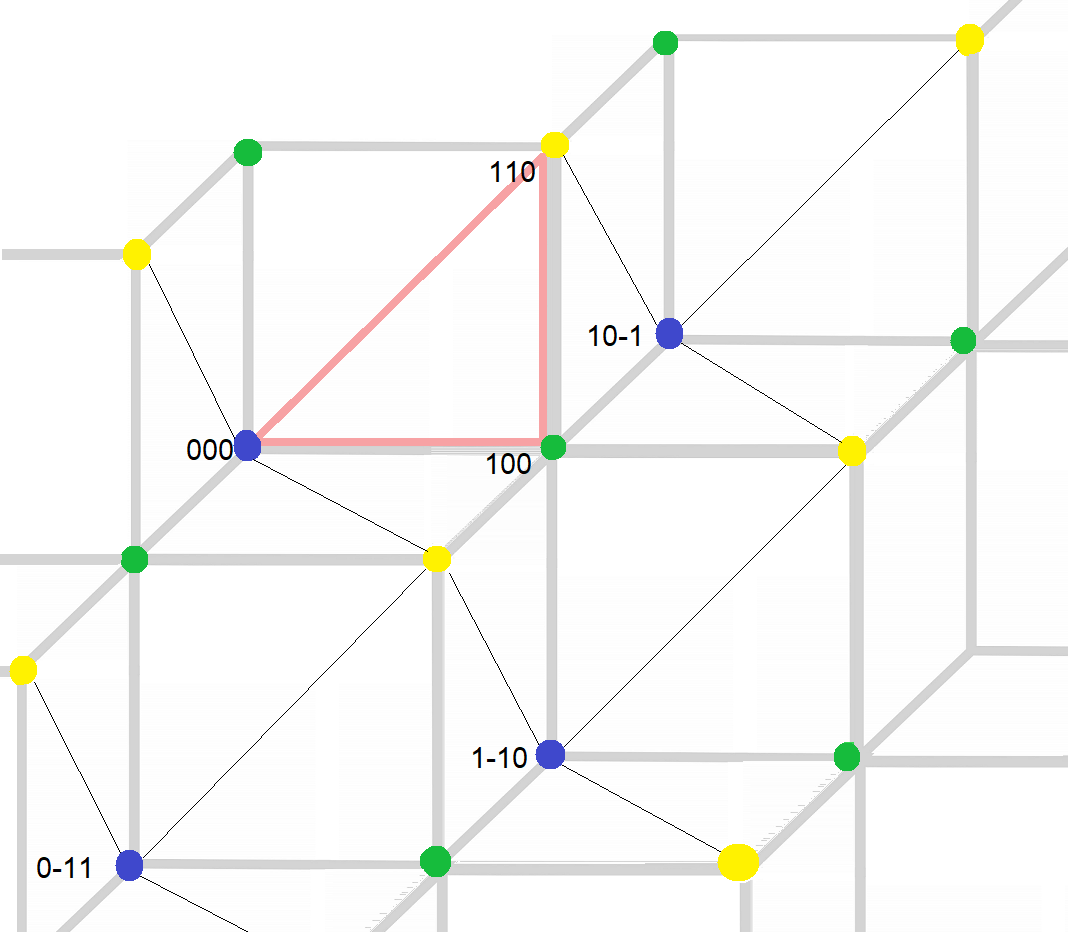}
\caption{Triangulation of the stair lattice}
\label{fig:stair}
\vspace{1cm}
\end{figure}
Let $\bbA_0 \subset \bbZ^{n}$ be the set of all integer points $a\in \bbZ^{n}$ such that $\sum_j a_j = 0$. Define similarly the lattices $\bbA_i$ for $i\in \{0,1,\ldots,n-1\}$.\\
It follows by definition (\ref{def:afper}) that every element $u\in \tilde{S}_n$ is of the form 
$$[r_1 + a_1\cdot n,\: \ldots\: , r_{n}+ a_{n} \cdot n]$$ 
where $(a_1,\ldots,a_{n})\in \bbA_0$ and $[r_1,\ldots, r_{n}]\in S_n$, i.e., $\{r_1,\ldots, r_{n}\} = \{1,\ldots, n\}$.\\
Thus it's not hard to see that $\bbA_0$ is a normal subgroup of $\tilde{S}_n$, and one has an isomorphism 
$$
\tilde{S}_n \cong S_n \rtimes \bbA_0
$$
Where the action of $S_n$ on $\bbA_0$ in the semidirect product definition is by permutation of coordinates.\\
Now, there's a natural isomorphism between the semidirect product and $G_n$. To see this clearly, consider the action of $G_n$ on the appropriate triangulation of the union of lattices $\bbA_0 \cup\ldots \cup \bbA_{n-1}$,  name it "the stair lattice". Explicitly, consider the origin $0^n$. Every permutation $\sigma\in S_n$ corresponds to a simplex which is the convex hull of the points 
$$0^n\textbf{,}\:\: 0^n + e_{\sigma(1)}\textbf{,}\:\:  0^n + e_{\sigma(1)} + e_{\sigma(2)}\textbf{,}\:\:\ldots\:\: \textbf{,}\:\: 0^n + \sum_{i=1}^{n-1} e_{\sigma(i)}$$ 
And it's easy to see that every involution $f_i$ for $ i>0$ is expressed by swapping $e_{\sigma(i)}$ with $e_{\sigma(i+1)}$, i.e., it swaps between neighbouring simplices oriented by $\sigma$ and $(i,i+ 1)\sigma$. This shows that $F_0\cong S_n$. The action of $f_0$ on a $\sigma$-oriented simplex moves $0^n$ to $0^n+ e_{\sigma(1)} - e_{\sigma(n)}$, and the simplex, whose vertices are $0^n, 0^n+e_{\sigma(1)}, \ldots, 0^n+\sum_i e_{\sigma(i)}$ to 
$$
(0^n+e_{\sigma(1)} - e_{\sigma(n)})\textbf{,}\:\: (0^n + e_{\sigma(1)} -e_{\sigma(n)}) + e_{\sigma(n)}\textbf{,}\:\: (0^n + e_{\sigma(1)} -e_{\sigma(n)}) + e_{\sigma(n)} + e_{\sigma(2)}\textbf{,}\: \ldots
$$
i.e., $f_0$ swaps between neighbouring simplices oriented by $\sigma$ and $(n,1)\sigma$. Thus, $F_i\cong S_n$ for every $i$.

See figure \ref{fig:stair} for an illustration. (That's equivalent to simply labeling the points in $\bbZ^{n-1}$ as in the stair lattice).\\

Now, label the set of simplices as follows: for every simplex with orientation $\sigma\in S_n$ around the lattice point $a\in \bbA_0$, label it by first translating $[1,2,...,n]$ by $a$ and then permuting by $\sigma$, i.e., 
$$
[\sigma(1)+a_{\sigma(1)}\cdot n,\ldots , \sigma(n)+a_{\sigma(n)}\cdot n]
$$
Now clearly, for $0<i\leq n-1$ the action of $f_i$ on the simplices coincides with the action of $\tilde{s}_i\in \tilde{S}_n$ on $\bbZ$. Finally, to see that $f_0$ coincides with $\tilde{s}_0$, consider the $[1,2,\ldots,n]$-oriented simplex around $0^n$. $f_0$ moves it to the $[n,2,\ldots,n-1,1]$-oriented simplex around the point $(1,0,\ldots,0,-1)$ (see remark below). By our definition, the label of this simplex is 
$$
[n +(-1)\cdot n, 2,\ldots,n-1, 1+1\cdot n] = [0, 2, \ldots, n-1, n+1]
$$
as expected. One can verify in a similar manner the action on other simplices.
\begin{remark}
The set of simplices can be labeled (and the affine permutation group can be represented) by any window of width $n$. Here we used $[1,\ldots,n]$ while earlier we implicitly used the window $[0,\ldots,n-1]$. A matter of convenience.
\end{remark}
\paragraph{Combinatorial description}
\label{subsub:comb}
Here we give a very simple description of the graph before the symmetrization; i.e. we describe an infinite graph whose vertices are the integer lattice. It can be made finite by restricting to tori. The skeleton of the clique complex of the resulting graph is type-regular, and thus can be symmetrized.

Following the observations above, consider again the triangulation of $\bbR^n$ as the proper graph $G$, i.e., The vertices are the lattice points and the edges are all pairs $x,y$ that lie in a common simplex. Consider the cube $[0,1]^{n-1}$. Every simplex corresponds to an increasing path from $0^{n-1}$ to $1^{n-1}$. Therefore, two points $x,y$ are neighbors if they lie in a common cube and there's an increasing path (w.r.t the cube structure) from one of them to the other; In other words, $x\sim y$ if $(x-y)\in \{0,1\}^{n-1} \cup \{0, -1\}^{n-1}$. The description in neater when considering this graph on the stair lattice: Let 
$$ \bbA = \bbA_0 \cup \ldots \cup \bbA_{n-1} \subset \bbZ^{n}
$$ 
Now, define a graph $G' = (\bbA,\:\: E'= \bigcup\limits_{0\leq i<j\leq n-1} E_{i j})$ where for every $0\leq i<j\leq n-1$
$$
E_{i j} = \big\{\{x,y\}:\: x\in \bbA_i, y\in \bbA_j; \: (y-x) \in \{0,1\}^{n} \text{  and has precisely } (j-i) \text{ 1's} \big\}
$$
For a finite graph, one can take the stair lattice of the torus $\bbZ^{n}_{n\cdot k}$.
A small variation of this (example \ref{subsub:3r}) yields a $(2\binom{3r}{r},\binom{2r}{r})$-regular graphs for every integer $r$.
\begin{example}[Graph degree]
Let $x\in V_0$. $x$ has $\binom{n}{j}$ neighbors in $V_j$ (as this is the number of ways to add $j$ 1's to $x$), thus 
$$deg(x) = \sum_{j=1}^{n-1} \binom{n}{j}= 2^{n}-2$$.
\end{example}

\subsubsection{Other Coxeter constructions}
\begin{figure}
\centering
\includegraphics[scale=0.6]{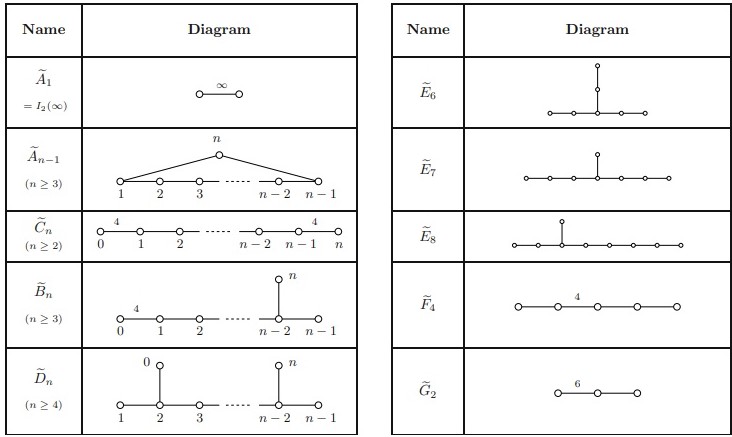}
\caption{The affine irreducible Coxeter systems}
\label{fig:diag}
\vspace{1cm}
\end{figure}
The type $\tilde{A}_n$ Coxeter construction of subgroup geometry system is a special case that can be generalized: Let $\langle W, S \rangle$ be a Coxeter system \cite{coxbook}, where $W$ is a group and $S = \{s_0,s_1,\ldots,s_n\}$ is a set of generators of $W$. Define for every $i\in \{0,1,..,n\}$, $K_{\{i\}}:= \langle\{s_j: j\neq i\} \rangle$.
\begin{claim}
$(W, (K_{\{i\}})_{i\in \{0,\ldots,n\}}))$ is a subgroup geometry system.
\end{claim}
\begin{proof}
 (A1), (A2) follows from theorem 5.2 in \cite{cohen}. (A3) is obvious since $K_I = \{e\}$ and $K_{I\setminus \{i\}} = \langle \{s_i\}\rangle$. (follows as well from theorem 5.2 in \cite{cohen}).
\end{proof}

By taking affine Coxeter systems (which are infinite) we can generate infinite families of HRG in a similar way, since each contains a normal abelian subgroup such that the corresponding quotient group is finite.\\
Unfortunately, non of these subgroup geometry systems have good expansion in their $1$-dimensional links: It can be seen in figure $\ref{fig:diag}$
(A Coxeter diagram encodes the corresponding Coxeter system: The vertices of a diagram are the group's generators, which are all of order 2; There's an unlabeled edge between $g,h$ if $gh$ is of order $3$; There's a labeled edge by $j\geq 4$ between $g,h$ if $gh$ is of order $j$). Note that every system has $S_{3}$ as subgroup, spanned by two adjacent generators - i.e., in every similar Coxeter construction we will have $C_6$ as a 1-dimensional link, and, as mentioned in \ref{par:afexp} this implies that the upper bound on the expansion parameter obtained by \ref{lem:trick} is trivial.

\subsection{Ad-hoc constructions}
\label{sub:ad}
Throughout this sub-subsection, we use the groups presented earlier and their subgroups as building blocks. In order to generate infinite family from a certain construction, one has to take an appropriate quotient (as seen earlier) and apply (in the relevant constructions) the $\circledast$ product on the resulting finite graphs. 

\subsubsection{A $(2\binom{3r}{r}, \binom{2r}{r})$-regular family}
\label{subsub:3r}
Here is a family of  $(a,b)$-graphs whose links are (a bipartite version of) a Kneser graph.
Based on the combinatorial description in \ref{subsub:comb}: Let $r \geq 1, k\leq 1$.  
Define $G' = (V'= V_0\cup V_r \cup V_{2r},\: E')$ where
$$
V_j = \{x\in \bbZ_{3r\cdot k}^{3r}: \sum_i x_i = j \}
$$
and
$$
E' = \big\{\{x,y\}: (x-y) \in \{0,1\}^{3 r} \text{  and has precisely r 1's} \big\}
$$
Consider for example a vertex $x\in V_r$. There are $\binom{3r}{r}$ ways of adding $r$ 1's to $x$ to reach some vertex in $V_{2r}$, and there are $\binom{3r}{r}$ ways of subtracting $r$ 1's to reach $V_0$. So, the degree of each vertex is $2\binom{3r}{r}$, and similarly it's easy to see that the degree of the links is $\binom{2r}{r}$. Furthermore, the induced graph on the neighbourhood of a vertex is the bipartite Kneser graph $K_2\otimes K(3r, r)$ (Unfortunately, it doesn't have good expansion properties, since the graph $K(3r,r)$ is $m$ regular, with second eigenvalue precisely $m/2$, where $m =\binom{2r}{r}$).

\subsubsection{Improving the blow up}
One can notice that the use in the full permutation group in the symmetrization causes a massive blow up in the degrees. Ideally, given a partite type-regular graph, we would want to make the least possible number of product action to make it hyper-regular. Actually, for an arbitrary partite type-regular graph, by the way of the proof of \ref{cla:ptr}, it follows that the permutation group only has to be set transitive:

\begin{definition}[i-set transitive]
A group $H\leq S_n$ is i-set transitive if for every pair of subsets $S,T$ of $N=\{1,2,\ldots,n\}$. each containing i elements, there exists a permutation in $H$ which carries $S$ into $T$.
\end{definition}

\begin{definition}[set transitive]
A group $H\leq S_n$ is set transitive if $H$ is i-transitive for every $1\leq i \leq n-1$.
\end{definition}
Clearly the symmetric group $S_n$ is set-transitive, and the alternating group $A_n$ is set-transitive for $n > 2$. 
However, it turns out that, except for a finite number of exceptions, these are the only examples. 
R.  A.  Beaumont  and  R.  P.  Peterson proved in \cite{peterson} the following:

\begin{theorem}
A group $H \leq S_n$ on n symbols, which does not contain the alternating
group $A_n$ is not set-transitive, with the
exceptions of n = 5, 6, and 9. 
\end{theorem}
In general, without further assumptions on the graph symmetries, the best we can do is to use $A_n$, which is of order $\frac{n!}{2}$. However, if one seeks to build an $(a,b,c,d)$-regular graph with reasonable degree blow up, for $n=5$ there exists a set transitive permutation group of order 10, for example:
$$
H= \langle (12345), (12)(35) \rangle
$$
(Note that $|S_5|= 120, |A_5| = 60$).

\subsubsection{Exploiting group structure}
\label{sec:ad}
Our two main examples happen to have a rich automorphism group. Let's focus on $G=EL_{n+1}(\cR)$. 
Kaufman and Oppenheim showed (Theorem 3.10 in \cite{kaufman2017simplicial}) that there's an embedding $\Gamma \xhookrightarrow{} Aut(G)$ where $\Gamma\cong D_{n+1}$ (the Dihedral group on $n+1$ elements), which acts on the set of subgroups as follows
$$
\forall \gamma\in D_{n+1}, \: \gamma.K_{\{i\}} = K_{\{\gamma(i)\}}
$$
In particular, since $D_{n+1} < Aut(G)$, for every $\tau \subset I,\: \gamma\in D_{n+1}$: $K_{\tau} \cong K_{\gamma(\tau)}$.
\begin{remark}
The above holds for type $\tilde{A}_n$ Coxeter group as well and follows from its cyclic diagram (figure \ref{fig:diag}). 
\end{remark}
We saw in \ref{cla:tran} that a subgroup geometry system graph is hyper-regular if for every $m\in [n+1]$, every two intersections of $m$ different subgroups are isomorphic. In other words, it's sufficient that $Aut(G)$ contains a subgroup that acts set-transitively on $\{K_{\{i\}}: i\in I\}$.

\begin{example}
For $n=2$, $D_3 \cong S_3$, and indeed, both of our examples (\ref{sec:elem}, \ref{sec:triang}) yield $(a,b)$-regular graphs without additional symmetrization. In \ref{sec:triang}, for $n=2$ the construction yielded a $(6,2)$-regular graph.
\end{example}

Note that if $D_{n+1}\leq Aut(G)$ and acts on $\{K_{\{i\}}: i\in I\}$ as on $I$, then $D_{n+1} \times D_{n+1} \rtimes S_2 \leq Aut(G\times G)$ and acts on $\{K_{\{i\}}: i\in I\}\times \{K_{\{i\}}: i\in I\}$ as on $I\times I$, and so on.\\
This leads us to the following question: Given $n>1$, we ask what is the smallest $m>0$ for which there exists $n+1$ subgroups in $\{K_{\{i\}}: i\in I\}^{m} \leq G^m$, such that  
$$
D_{n+1}^{m}  \rtimes S_m \leq Aut(G^m) 
$$
acts set transitively on them. Here are some ad hoc examples:
\begin{itemize}
    \item $n = 3$: ($m=3$)
$$
    K_0 \times K_0 \times K_0
$$ $$
    K_1 \times K_2 \times K_3
$$ $$
    K_2 \times K_3 \times K_1
$$ $$    
    K_3 \times K_1 \times K_2
$$
I.e., 3 permutations ($Id, (1,2,3), (1,3,2)$) instead of $4!=24$.
    \item $n = 4$: ($m=2$)
    \begin{figure}
    \centering
    \includegraphics[scale=0.3]{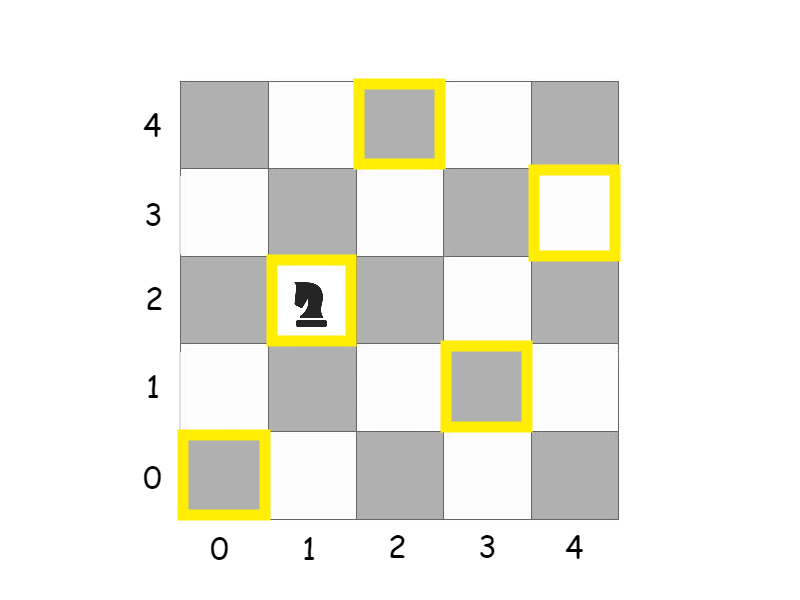}
    \caption{A knight-cycle on the $5\times 5$ torus of length 5. All subsets of same size of the five points are equivalent, enabling a simple construction of  link-connected $(200,36,12,4)$-graphs.}
    \label{fig:cycle}
    \vspace{1cm}
    \end{figure}
$$
    K_0 \times K_0
$$ $$
    K_1 \times K_2 
$$ $$
    K_2 \times K_4
$$ $$    
    K_3 \times K_1
$$ $$
    K_4 \times K_3 
$$

I.e., 2 permutations ($Id, (1,2,4,3)$) instead of $5!=120$.
\end{itemize}
Note that a group is i-set transitive if and only if it's $n-i$ set transitive. It's easy to verify that $D_{5} \times D_{5} \rtimes S_2$ acts 1-set and 2-set transitively on $\{{(0,0), (1,2), (2,4), (3,1), (4,3)}\}$, and thus, acts set transitively on it. (See figure \ref{fig:cycle}). 



\subsection{Explicit degrees}
Here we calculate the degrees of the main HRG's presented in this section. \\
Generally, let $(G, (K_{\{i\}})_{i\in I})$ be subgroup geometry system and let $X$ be its coset geometry. For $\tau \subsetneq I$, the size of the link of $K_{\tau}$ is $\sum_{i\in I\setminus \tau} \frac{|K_{\tau}|}{|K_{\tau\cup {\{i\}}}|}$. If $X$ has property \ref{cla:tran}, then it is hyper-regular and the size of the link of every $k< n$ dimensional simplex is $$
(n-k)\frac{|K_{\{0,\ldots,k\}}|}{|K_{\{0,\ldots,k,k+1\}}|}
$$\\
Thus, for the symmetrization of X, $X^{\circledast S_{I}}= (G^{(n+1)!}, (\prod_{\pi \in S_I} K_{\{\pi{(i)}\}})_{i\in I})$, every link of a $k$-dimensional simplex is of size 
$$
\sum_{j>k}\dfrac{|\prod_{\pi \in S_{I}} K_{\pi(\{0,\ldots,k\})}|}{|\prod_{\pi \in S_{i}} K_{\pi(\{0,\ldots,k,j\})}|} = (n-k) \prod_{\pi \in S_{I}} \frac{|K_{\pi(\{0,\ldots,k\})}|}{|K_{\pi(\{0,\ldots,k, k+1\})}|}
$$
And by the symmetry of the construction
$$
= (n-k) \bigg[\prod_{{\substack{|\tau|=k+1; \\ \tau\subset I\setminus \{b\}}}} \frac{|K_{\tau}|}{|K_{\tau \cup \{b\})}|}\bigg]^{\frac{(n+1)!}{\binom{n}{k+1}}} 
$$
for a fixed $b\in I$, since every $\frac{|K_{\tau}|}{|K_{\tau\cup {\{b\}}}|}$ appears the same number of times in the product, and the number of elements in the product inside the squared paernthesis is ${|\{\tau\in I\setminus{\{b\}}:\:|\tau|=k+1\}|} ={\binom{n}{k+1}}$.

\subsubsection{Elementary matrices groups}
We'll need a more explicit characterization of the subgroups. Reminder: $T= \{r_0 + r_1\cdot t: r_0, r_1 \in \bbF_q\}$. Let $T^{k} = \{r_0+ \sum_{i=1}^k r_i t^i: r_i\in \bbF_q\}$ and note that $|T^k| = q^{k+1}$.\\ Here's a version of Corollary 3.3 in \cite{kaufman2017simplicial}:
\begin{lemma}
Let $0\leq p< n $ and $0\leq a_0<\ldots<a_p \leq n$. For $a\neq b\in I$, Let $Arc(a, b)$ denote the set $\{a+1, a+2,\ldots,b-1\}$ (each element is taken $\mod n+1$ ) and as convention, let $Arc(a,a) = \{a+1, a+2,\ldots, a-1\}$. For $\tau=\{a_0, a_1,\ldots,a_p\}$, $K_\tau$ is the group composed of all the matrices $M=(M_{k,j})$ such that
$$
    M_{k,j}\in  
\begin{cases}
    \{1\}   & k=j\\
    T^{j-k} & k\neq j; \exists i \text{ s.t } \{k, k+1,\ldots,j-1\} \subseteq Arc(a_i, a_{i+1}) \\
    0,      & \text{otherwise}
\end{cases}
$$
where $j-k$ and $k, k+1, \ldots, j-1$ above are taken $\mod (n+1)$, and $i+1 \mod p+1$.
\end{lemma}
Define, for $a\neq b \in I$
$$
f(a,b) = \sum_{\substack{k\neq j; \\ \{k, k+1,\ldots,j-1\} \subseteq Arc(a, b)}}\:  1+ (j-k \mod n+1)
$$
Then,
$$
|K_\tau| = q^{\sum_{i=0}^p f(a_i, a_{i+1})}
$$
Let $b\in I\setminus \tau$ and let $i$ be s.t $b\in Arc(a_i, a_{i+1})$, then
$$
\frac{|K_{\tau}|}{|K_{\tau\cup {\{b\}}}|} = q^{f(a_i,a_{i+1}) - f(a_i, b) - f(b, a_{i+1})}
$$
Clearly, the value of $f(a,b)$ depends only on the difference between $b$ and $a$. Hence, we define $g: [n] \rightarrow \bbN$: for every $a,b$ such that $|Arc(a,b)| = m\in [n]$, $g(m) := f(0, m+1) = f(a,b)$. Now,
$$
g(m) = \sum_{i=0}^{m-1} (i+2)(m-i) = \frac{m^3+6m^2+5m}{6}
$$
\newline
\paragraph{Examples}
\begin{itemize}
    \item \textbf{$(a,b)$-regular:} Take $EL_{3}(\bbF_{q}[t])$. As illustrated nicely in Dinur's notes \cite{HDX} the 2-dimensional coset geometry that generated by this group as described earlier, is $(2 q^{5}, q^2)$-regular.
    \item \textbf{$(a,b,c,d)$-regular:} Take the $4$-dimensional construction shown in \ref{sec:ad}, i.e., let \newline
$$X=X(G, (K_{\{0\}} \times K_{\{0\}},K_{\{1\}} \times K_{\{2\}}, K_{\{2\}} \times K_{\{4\}}, K_{\{3\}} \times K_{\{1\}}, K_{\{4\}} \times K_{\{3\}}))$$ \newline
For both our subgroup geometry systems, $X$ is \newline
$$\bigg(4\cdot \frac{|K_{\{0\}}|}{|K_{\{0,1\}}|} \frac{|K_{\{0\}}|}{|K_{\{0,2\}}|}, 3\cdot \frac{|K_{\{0,1\}}|}{|K_{\{0,1,2\}}|} \frac{|K_{\{0,2\}}|}{|K_{\{0,2,4\}}|}, 2\cdot \frac{|K_{\{0,1,2\}}|}{|K_{\{0,1,2,3\}}|} \frac{|K_{\{0,2,4\}}|}{|K_{\{0,2,4,1\}}|}, 1\cdot \frac{|K_{\{0,1,2,3\}}|}{|K_{\{0,1,2,3,4\}}|} \frac{|K_{\{0,2,4,1\}}|}{|K_{\{0,2,4,1,3\}}|}  \bigg)$$ \newline
regular (one can verify). Specifically, take $EL_5(\bbF_q[t])$. The degrees are \:\: $$\bigg(4\cdot q^{g(4)-g(3)} q^{g(4)-g(1)-g(2)}, 3\cdot q^{g(3)-g(2)} q^{g(2)-g(1)} , 2\cdot q^{g(2)-g(1)} q^{g(1)}, 1\cdot q^{g(1)} q^{g(1)}  \bigg) = (4 q^{35}, 3 q^{14}, 2 q^{7}, q^{4})$$
    \item \textbf{$(d_0, \ldots, d_{n-1})$-regular:} Symmetrization construction using $S_{(n+1)}$. For $1\leq k \leq n-1$,  the size of the link of every $k$-dimensional simplex
\newline
$$
d_{k} = (n-k)\bigg[\prod_{m=1}^{n-k} \Big( \prod_{j=1}^m q^{f(0,m+1) - f(0,j) - f(j,m+1)}\Big)^{\binom{n-m-1}{k-1}} \bigg]^{ \frac{(n+1)!}{\binom{n}{k+1}}}
$$
\newline
Inside the squared Parenthesis, we first calculate 
$$\prod_{\substack{|\tau|=k+1; \\ \tau\subset I\setminus \{b\}}} \frac{|K_{\tau}|}{|K_{\tau\cup {\{b\}}}|}$$
for a fixed $b$, by going over all arcs that $b$ can be in, i.e., for every $1\leq m \leq n-k$, $b$ can be inside $m$ different arcs of size $m$.  Every arc dictates $2$ out of the $k+1$ elements in $\tau$, so we have $\binom{n-m-1}{k-1}$ ways to complete it.\\
\\Let's simplify this expression. First focus on the exponent of $q$:
$$
\sum_{j=1}^{m} g(m) -(g(j-1)+g(m-j)) = \big[ m\cdot g(m) -2 \sum_{j=1}^m g(j-1) \big]= \frac{m^4+6m^3+11m^2+6m}{12}
$$
So
\newline
$$
d_{k} = (n-k)\cdot q^{{ \frac{(n+1)!}{\binom{n}{k+1}}} \sum_{m=1}^{n-k} \binom{n-m-1}{k-1} \frac{m^4+6m^3+11m^2+6m}{12}}
$$
\newline
(Sanity check: indeed, $d_{n-1}=q^{2 (n+1)!}$). \\
And $d_0$ is a bit different
$$
d_0 = n\cdot q^{{ \frac{(n+1)!}{n}} \sum_{j=1}^{n} g(n) - (g(j-1)+g(n-j))} =  n\cdot q^{{ (n+1)!} \frac{n^3+6n^2+11n+6}{12}}
$$
    \end{itemize}

\subsubsection{Affine permutation group}
Let $i, j\in I$ where $i<j$.
$$
K_{\{j\}} \cong K_{\{i\}} \cong S_{(n+1)}; \:\:\:\:\:\:\:\:\: K_{\{i,j\}} \cong S_{(j-i)}\times S_{(n+1 - (j-i))}
$$
Thus
$$
\frac{|K_{\{i\}}|}{|K_{\{i,j\}}|} = \frac{|K_{\{j\}}|}{|K_{\{i,j\}}|} = \frac{(n+1)!}{(j-i)!(n+1 - (j-i))!} = \binom{n+1}{j-i}
$$
\newline
Now, for $1\leq k \leq n-1$, let $a_0<a_1<\ldots<a_k \in I$.
$$
K_{\{a_0,\ldots,a_k\}} \cong \langle\{(i, i+1 \mod n+1): i \notin \{a_0,\ldots,a_k\} \}\rangle \cong  S_{(a_1-a_0)}\times S_{(a_2-a_1)}\times \ldots \times S_{(a_0-a_p \mod n+1)}
$$
Let $b\in I$. If $b\in (a_j, a_{j}+1,\ldots, a_{j+1})$, then
$$
\frac{|K_{\{a_0,\ldots,a_p\}}| }{|K_{\{a_0,\ldots,a_j,b,a_{j+1},\ldots,a_p\}}| } = \frac{(a_{j+1} - a_j)!}{(b-a_j)!(a_{j+1}-b)!} = \binom{a_{j+1}-a_j}{b-{a_j}}
$$\newline
Where ${a_{j+1}-a_j}, {b-{a_j}}$ are taken modulo $n+1$ if needed.
\newline
\paragraph{Examples}
\begin{itemize}
    \item \textbf{$(a,b)$-regular:} $(\tilde{S}_{3r}, (K_{\{0\}}, K_{\{r\}}, K_{\{2r\}}))$ is subgroup geometry system (see \ref{lem:affine}), thus the degrees of $X(\tilde{S}_{3r}, (K_{\{0\}}, K_{\{r\}}, K_{\{2r\}}))$ are
    $$\Big(\frac{K_{\{r\}}}{K_{\{0, r\}}} + \frac{K_{\{r\}}}{K_{\{r, 2r\}}},\: \frac{K_{\{0, r\}}}{K_{\{0,r,2r\}}}\Big)\:=\:(2 \binom{3r}{r}, \binom{2r}{r})$$
    (same as the combinatorial construction in \ref{subsub:3r}).

    \item \textbf{$(a,b,c,d)$-regular:} Take $G= \tilde{S}_5$ with subgroup construction as in \ref{sec:ad} (for $n=4$). The degrees are
$$ \:\: \bigg(4\cdot \binom{5}{1} \binom{5}{2}, 3\cdot \binom{4}{1} \binom{3}{2} , 2\cdot \binom{3}{1} \binom{2}{1}, 1\cdot \binom{2}{1} \binom{2}{1}  \bigg) = (200, 36, 12, 4)$$
Actually, for $r\geq 1$, a similar construction using $\tilde{S}_{5 r}$ and $K_{\{i \cdot r\}}$ for $i=0,1,\ldots,4$, yields a
$$ \:\: \bigg(4\cdot \binom{5r}{r} \binom{5r}{2r}, 3\cdot \binom{4r}{r} \binom{3r}{2r} , 2\cdot \binom{3r}{r} \binom{2r}{r}, 1\cdot \binom{2r}{r} \binom{2r}{r}  \bigg)$$
regular graph.
\item \textbf{$(d_0, \ldots, d_{n-1})$-regular:} Symmetrization construction using $S_{(n+1)}$.
In a similar way to the previous case,
$$
d_0 = n \cdot \Big[\prod_{j=1}^n \binom{n+1}{j}\Big]^{\frac{(n+1)!}{n}}
$$
and for $1\leq k \leq n-1$
$$
d_k = (n-k)\bigg[\prod_{m=1}^{n-k} \Big[ \prod_{j=1}^m \binom{m+1}{j}\Big]^{\binom{n-m-1}{k-1}} \bigg]^{ \frac{(n+1)!}{\binom{n}{k+1}}}
$$
\\(Sanity check: indeed, $d_{n-1}=2^{(n+1)!}$)
\end{itemize}

\section{Further questions}
\begin{enumerate}
    \item \label{sum:1} We saw that the $\circledast$ product preserves the expansion of connected biregular graphs. That's not always true for $n$-partite products with $n > 2$: In particular, when $n>2$, the product of connected graphs is not necessarily connected. We know that a sufficient condition for the product to be connected is if the graphs are pure. In that case, what can we say about the expansion of an $n$-partite product? i.e., given $G_1, G_2$ pure $n$-partite clique complexes, is there a general upper bound for $\lambda_2(G_1\circledast G_2)$, tighter than the one obtained by the Trickling Down Theorem?
    \item Is there a probabilistic construction of type-regular expanders? this will yield, of course, a probabilistic construction of hyper-regular expander. 
\end{enumerate}
\medskip
 \section{Acknowledgements}
 The authors wish to thank both Uri Bader and Alex Lubotzky for their patient and lucid explanations, and Yotam Dikstein and Irit Dinur for introducing them to the topic.

\end{document}